\documentclass[reqno]{amsart}

\usepackage{amsmath,amsthm,amstext}
\usepackage{amsthm,amsfonts,amssymb,euscript,mathrsfs,graphics,color,amsmath,amssymb,latexsym,marginnote}

\usepackage{hyperref}
\usepackage{setspace,url}
\usepackage{xcolor}
\usepackage{bm,bbm}

\usepackage{scalerel,stackengine}
\stackMath
\newcommand\reallywidecheck[1]{%
\savestack{\tmpbox}{\stretchto{%
  \scaleto{%
    \scalerel*[\widthof{\ensuremath{#1}}]{\kern-.6pt\bigwedge\kern-.6pt}%
    {\rule[-\textheight/2]{1ex}{\textheight}}
  }{\textheight}%
}{0.5ex}}%
\stackon[1pt]{#1}{\scalebox{-1}{\tmpbox}}%
}

\usepackage[margin=1.2in]{geometry}

\newtheorem{theorem}{Theorem}[section]
\newtheorem{lemma}[theorem]{Lemma}

\newtheorem{proposition}[theorem]{Proposition}
\newtheorem{remark}[theorem]{Remark}

\numberwithin{equation}{section}

\newcommand{\A}{Q_a}
\newcommand{\B}{Q_b}
\newcommand{\C}{Q_c}
\newcommand{\sgn}{\mathrm{sgn}}

\renewcommand{\v}{v}
\newcommand{\q}{q}

\renewcommand{\H}{\mathcal{H}}

\renewcommand{\part}[1]{{\bf Part #1}}

\renewcommand{\Im}{\text{Im}}

\newcommand{\R}{\mathbb{R}}
\newcommand{\f}{\beta}

{\begin{trivlist} \item[]{\textbf{Proof} }}%
{\hspace*{\fill}$\rule{.3\baselineskip}{.35\baselineskip}$\end{trivlist}}

\newcommand{\ak}[1]{\color{red}{#1} \color{black}}

\makeatletter
\def\author@andify{%
  \nxandlist {\unskip ,\penalty-1 \space\ignorespaces}%
    {\unskip {} \@@and~}%
    {\unskip \penalty-2 \space \@@and~}%
}
\makeatother

\begin{document}

\title{Interaction between long internal waves and free surface waves in deep water}

\author{Adilbek Kairzhan}
\author{Christopher Kennedy}
\author{Catherine Sulem}

\address[A. Kairzhan]{Department of Mathematics, Nazarbayev University, 010000, Astana, Kazakhstan.}
\email{akairzhan@nu.edu.kz}
\address[C. Kennedy]{Department of Mathematics and Statistics, Queen's University, K7L 3N6 Kingston Ontario, Canada. }
\email{christopher.kennedy@queensu.ca}
\address[C. Sulem]{Department of Mathematics, University of Toronto, M5S 2E4 Toronto Ontario, Canada.}
\email{sulem@math.utoronto.ca}

\thanks{}
\thanks {}
\date{\today}

\begin{abstract}
We consider a  density-stratified fluid composed of two immiscible layers  separated by a sharp interface. 
We study the regime  of long  internal waves interacting with modulated surface  wave packets and
 describe  their  resonant interaction  by a system of equations 
where the internal wave solves  a high-order Benjamin-Ono  (BO) equation coupled to a linear
Schr\"odinger equation for the envelope of the free surface.
The perturbation methods are based on the Hamiltonian formulation for the original system of irrotational Euler’s equations as described in Benjamin-Bridges \cite{BB97} and Craig-Guyenne-Kalisch \cite{CGK05}.
We also establish a local wellposedness result for a reduced BO-Schr\"odinger system using an  approach developed by  Linares-Ponce-Pilod  \cite{LPP11}.
\end{abstract}
\maketitle

\section{Introduction}
Oceans  are  stratified into layers of differing densities due to temperature or salinity gradients, which induce  internal waves, or waves that propagate along the interface between the layers. Thermoclines, or temperature gradients, commonly occur in tropical seas while pycnoclines, otherwise known as salinity gradients, are typically found in fjords. 
Understanding the formation and propagation of internal  waves  has been the object of intense studies due to their importance in oceanography and
geophysical fluid dynamics. 
Internal waves can affect measurements of currents and undersea navigation as well as offshore construction structures by imposing significant amount of stress. They also have an influence on the mixing of different layers of water in the ocean. 

Due to the technological progress over the past several decades, it has become possible to observe and accurately measure internal waves in the oceans. One of the early measurements have been taken by Perry and Schimke \cite{PS65} in the Andaman Sea, where internal waves of 80 meters high with wavelengths of 2000 meters and a thermocline situated at roughly 500 meters deep in the 1500 meters deep sea were found. Surprisingly, as it was observed in \cite{PS65} and later by Osborne and Burch \cite{OB80}, the presence of internal waves results in the presence of short, choppy, small-amplitude waves or ``rips'' followed by calmness of the sea after its passage, the phenomena known as the ``mill pond" effect.  
The phenomena was already known a long time ago, without any scientific explanations, to Vikings as "dead water", where ships get trapped in the calm water and moved back and forth by underwater forces \cite{E04}. We refer to the work of Helfrich and Melville \cite{HM06} for a more detailed overview of internal waves.

The existence of solitary waves in internal layers of a stratified fluid was investigated in an early paper of Peters and Stoker \cite{PS60}. In this article, they considered a two-layer system of finite depth with a free surface and a free interface, and derived a criterion for the sign of a soliton-like wave at the interface. In \cite{B66}, Benjamin studied an analogous problem with a rigid lid boundary condition on the upper surface using the KdV model. 
The work was later extended by Benjamin \cite{B67} and Ono \cite{O75} to deep water case, who showed that the evolution of long internal waves follows the now well-known, Benjamin-Ono equation. Similar result in the framework of KdV equation was obtained by Benney \cite{Benney66} in shallow water for both the free surface and  rigid lid cases. Later in 1970s, Joseph \cite{J77} and Kubota, Ko and Dobbs \cite{KKD78} derived the Intermediate Long Wave equation to describe the propagation of long internal waves in case of a finite depth. Higher-order corrections to the above models were obtained by Craig, Guyenne and Kalisch in \cite{CGK05} using the framework of Hamiltonian perturbation theory. Their results recovered, in particular, an  extended KdV equation found in Kawahara \cite{K72}. 

The modeling theory for internal waves continued to develop assuming the presence of more physical variables. For example, the effect of a depth-dependent horizontal shear flow on internal waves  was studied by Grimshaw \cite{G81} in shallow and deep water cases. As a result, the KdV and Benjamin-Ono equations were derived with non-constant coefficients depending on a variable identified with a coordinate of a ray along which the wave propagates. Choi in \cite{C06} analyzed the interaction of a linear shear current with large amplitude internal waves. Under the Boussinesq regime of a small density jump between the layers, it was found that the direction of a shear current significantly affects the properties of a soliton-like wave. Khusnutdinova and Zhang in \cite{KZ16} considered, additionally, the effect of non-trivial geometry of waterfronts on the evolution of the surface and internal waves in three-dimensional fluid of finite depth. Assuming the shape of the waves is curvilinear, they derived a cylindrical KdV-type equation for the amplitude of the waves, and developed a theory describing the distortion of the curvilinear waves by the shear flow. 

A characteristic change in the reflectance of the water surface and the observed in \cite{PS65} "ripple effect" have provided empirical evidence of coupling between the surface and internal waves. 
The strength of the interaction depends on the relation between the following properties of the waves: (A) their length scales and (B) their group and phase velocities. The studies of (A) in the literature are mostly subdivided into two possibilities. The first is when the length scales of the surface and internal waves are comparable, e.g. Gear and Grimshaw \cite{GG84} and P\u{a}ra\u{u} and Dias \cite{PD01}. In \cite{GG84} the authors also considered the relation between the phase velocities of the wavefronts and observed the strong interaction in the case where the phase velocities are almost equal, which leads to a coupled KdV-KdV system. In the case where the phase velocities differ by a quantity $\mathcal{O}(1)$,  they observed the weak interaction leading to an uncoupled KdV-KdV system. The second possibility in studies of (A) is when the internal and surface waves have different length scales, which are usually described by a long internal wave interacting with modulated quasi-monochromatic surface waves. In the case of the finite depth, this leads to the coupled system composed of the  KdV equation, describing the evolution of the internal wave, and a Schr\"odinger equation for the evolution of the surface wave, e.g. Kawahara, Sugimoto and Kakutani \cite{KSK75}, Hashizume \cite{H80} and Funakoshi and Oikawa \cite{FO83}. All of these works studied the effect of the resonant interaction of the waves, when the group velocity of the short wave coincides with the phase velocity of the long wave, on the coupling of the system and the form of its solutions. In \cite{FO83} Funakoshi and Oikawa also derived the Benjamin-Ono -- Schr\"odinger system for the case of a deep fluid. Further analytical study of the resonant interaction of internal and surface waves in finite depth, including detailed mathematical description of the ``ripple" and ``mill pond" effects, was provided in Craig, Guyenne and Sulem \cite{CGS12} who showed that the surface signature is generated by a process analogous to radiative absorption.

 The main goal of this paper is to continue the study of the interaction between long internal waves and quasi-monochromatic free surface waves of a density-stratified two-layer fluid in deep water. Using the perturbation Hamiltonian theory, we derive a coupled system describing the evolution of the waves and explicitly provide the non-linear coupling coefficients, which are dependent on the physical parameters of the system. In our derivation, as a part of the reduction process, we naturally come up with the resonant condition when the group velocity of the surface wave coincides with the phase velocity of the longer-scale internal wave. 
Our results extends the previously obtained system of Funakoshi and Oikawa \cite{FO83} by including the contribution of higher-order terms. The latter is made possible by 
assuming that the amplitude $\varepsilon_1$ of the surface waves is smaller than the amplitude $\varepsilon$ of the internal waves and satisfies the relation $\varepsilon^{3/2} < \varepsilon_1 < \varepsilon$.

The two-dimensional fluid domain consists of two immiscible fluids separated by a free interface, which idealizes a sharp thermocline or pycnocline. We follow the approach known since Zakharov \cite{Z68}, in which the water wave equations are expressed as a Hamiltonian system.
The Hamiltonian formulation, for which the Hamiltonian identifies with the total energy  is written in terms of canonically conjugate variables $(\eta,\eta_{1}, \xi, \xi_{1})$ that corresponding to the elevations and traces of quantities related to  the velocity potentials of the interface and surface, respectively. Following the works of Benjamin-Bridges \cite{BB97},  Craig and Sulem \cite{CS93} and Craig, Guyenne and Kalisch \cite{CGK05}, we rewrite the Hamiltonian in terms of the Dirichlet-Neumann operators for the free interface $G(\eta)$ and surface $G_{1}(\eta,\eta_{1})$, which according to Coifman and Meyer \cite{CM85} are analytic near $\eta = \xi = \eta_{1} = \xi_{1} = 0$. This allows to expand the Hamiltonian functional $H$ accordingly. The quadratic part of $H$, denoted $H^{(2)}$, is derived in Section 3, while the cubic part $H^{(3)}$ is derived in Section 4. In Section 5, we introduce the scaling regime, which corresponds to a long internal wave and a near-monochromatic surface wave. Under this regime, we obtain the leading terms of $H^{(2)}$ and $H^{(3)}$. From the reduced Hamiltonian, we derive in Section 6 a higher-order Benjamin-Ono (BO) equation coupled to a linear Schr\"odinger equation describing the time evolution of the internal wave and surface wave envelope, respectively.  
In the last two sections, we establish a local well-posedness result for a reduced coupled BO-Schr\"odinger system. 
The main difficulties are the second-order derivatives in the nonlinear terms of the BO equation, which is taken care of by using a gauge transformation following the approach of Linares, Pilod and Ponce \cite{LPP11}, and the presence of coupling terms.

\section{Formulation of the problem}

\subsection{Euler equations} 
We consider a two-dimensional fluid domain composed of two immiscible fluids separated by a free interface $\{y = \eta(x,t) \}$ into lower and upper regions given by
\begin{equation*}
S(\eta) = \{(x,y): ~ x \in \mathbb{R},~  -\infty < y < \eta(x,t) \}, 
\end{equation*}
with the fluid density $\rho$, and
\begin{equation*}
S_{1}(\eta,\eta_{1}) = \{(x,y): ~ x \in \mathbb{R}, ~ \eta(x,t) < y < h_{1}+\eta_{1}(x,t) \},
\end{equation*}
with the fluid density $\rho_{1}$, respectively. We assume the system is in stable configuration  ($\rho  > \rho_{1}$). In such setting, 
the fluid motion is assumed to be a potential flow  with 
  velocities $\mathbf{u}(x,y,t) = \nabla \phi(x,y,t)$ in $S(t;\eta)$ and $\mathbf{u}_{1}(x,y,t) = \nabla \phi_{1}(x,y,t)$ in $S(t;\eta,\eta_{1})$
 in each fluid region.   
The  two  velocity potentials satisfy
\begin{equation}
\label{laplacian-phi}
\begin{aligned}
&\Delta \phi = 0, ~\, \mathrm{in} ~\, S(t;\eta) \\
&\Delta \phi_{1} = 0, ~\, \mathrm{in}~\, S_{1}(t;\eta,\eta_{1}).
\end{aligned}
\end{equation}
At the deep bottom, we assume the boundary condition $\phi(x,y) \to 0$ as $y \to -\infty$. 

At  the interface between the two fluid domains, we impose two  kinematic boundary conditions and the Bernoulli condition of balance of forces. Denoting $\widehat{\nu}$  the exterior unit normal pointing out of the free interface, the two equations addressing the kinematic conditions on the interface are
\begin{align}
\label{kinematic-condition-1-nu1}
\partial_{t}\eta &= \partial_{y}\phi - (\partial_{x}\eta)(\partial_{x}\phi) = \nabla \phi \cdot \widehat{\nu} \sqrt{1+|\partial_{x}\eta|^{2}}
\end{align}
and
\begin{align}
\partial_{t}\eta &=  \partial_{y}\phi_{1} - (\partial_{x}\eta)(\partial_{x}\phi_{1}) =  - \nabla \phi_{1} \cdot \widehat{\nu} \left (-\sqrt{1+|\partial_{x}\eta|^{2}} \right ).
\end{align}
The balance of forces implies  
\begin{align}
\rho \left (\partial_{t}\phi + \frac{1}{2}|\nabla \phi|^{2} + g \eta \right ) = \rho_{1} \left (\partial_{t}\phi_{1} + \frac{1}{2}|\nabla \phi_{1}|^{2} + g\eta \right ),
\end{align}
where $g$ is the  acceleration due to gravity. 

On the upper free surface $\{y = \eta_{1}(x)+h_{1} \}$, 
the velocity potential, $\phi_{1}$  and the surface elevation function, $\eta_{1}$, satisfy the kinematic condition
\begin{align}
\partial_{t}\eta_{1} = \partial_{y}\phi_{1} - (\partial_{x}\eta_{1})(\partial_{x}\phi_{1}) = 
\nabla \phi_{1} \cdot \widehat{\nu_{1}} \sqrt{1+(\partial_{x} \eta)^{2}},
\end{align}
where $\widehat{\nu_{1}}$  is the  unit exterior normal to the upper  the free interface,  
and the  Bernoulli condition
\begin{align}
\label{bernoulli-phi1}
\partial_{t} \phi_{1} + \frac{1}{2}|\nabla \phi_{1}|^{2} + g \eta_{1} = 0.
\end{align}

\subsection{Hamiltonian formulation} 
We introduce the canonical variables which allow to rewrite the equations of motion of the free interface and the free surface as the Hamiltonian system,  following Benjamin and Bridges \cite{BB97} and Craig, Guyenne and Kalisch  \cite{CGK05}.  
These involve the free boundaries $\eta$ and $\eta_1$, and their  canonically conjugated variables expressed in terms of the traces of the velocity potentials on the boundaries
\begin{equation*}
\begin{aligned}
&\xi (x,t) := \rho \phi (x, \eta(x,t), t) - \rho_1 \phi_1 (x, \eta(x,t), t),\\
&\xi_1 (x,t) := \rho_1 \phi_1 (x, h_1 + \eta_1(x,t), t).
\end{aligned}
\end{equation*}
Equations \eqref{laplacian-phi}-\eqref{bernoulli-phi1} 
take the following canonical form in terms of the variables $(\eta, \xi, \eta_1, \xi_1)$,
\begin{align}
\label{Hamiltonian-system-1}
\partial_{t} \begin{pmatrix} \eta \\ \xi \\ \eta_{1} \\ \xi_{1} \end{pmatrix} 
= \begin{pmatrix} 0 & 1 & 0 & 0 \\ -1 & 0 & 0 & 0 \\ 0 & 0 & 0 & 1 \\ 0 & 0 & -1 & 0 \end{pmatrix} \begin{pmatrix} \delta_{\eta}H \\ \delta_{\xi}H \\ \delta_{\eta_{1}}H \\ \delta_{\xi_{1}}H \end{pmatrix},
\end{align}
where the Hamiltonian $H$ is the sum of the kinetic energy $K$ and the potential energy $V$, 
\begin{equation}
\label{HamiltonianSum}
H = K+V.
\end{equation}
Here, the kinetic energy is the weighted sum of the gradients of the two potentials of the velocity flows
\begin{align}
\label{kinetic-energy}
K &= \frac{1}{2} \int_{\R} \int_{ -\infty}^{\eta(x)} \rho |\nabla \phi(x,y)|^{2}\, dy\, dx 
+ \frac{1}{2} \int_{\R} \int_{\eta(x)}^{h_{1}+\eta_{1}(x)} \rho_{1}|\nabla \phi_{1}(x,y)|^{2}\, dy\, dx,
\end{align}
while the potential energy is given by
\begin{align} \label{potential}
V = \frac{1}{2} \int_{\mathbb{R}} g(\rho-\rho_{1})\eta^{2}(x)\, dx + \frac{1}{2} \int_{\mathbb{R}} g \rho_{1} \left ((h_{1}+\eta_{1})^{2}(x) - h_{1}^{2} \right )\, dx.
\end{align} 
The flow of the Hamiltonian system \eqref{Hamiltonian-system-1} preserves the momentum, $I$, given by 
\begin{align}
\label{momentum}
I &:= \int_{\R} \left (\rho \int_{-\infty}^{\eta} \partial_{x}\phi \, dy + \rho_{1} \int_{\eta}^{h_{1}+\eta_{1}} \partial_{x}\phi_{1} \, dy \right )\, dx.
\end{align}
Indeed, one can verify that $I$ Poisson commutes with $H$.  

\subsection{ Dirichlet-Neumann operators}

We further rewrite the Dirichlet integrals of the kinetic energy \eqref{kinetic-energy} using the Dirichlet-Neumann operators (DNOs) for the two-fluid domains. The DNO for the lower fluid domain is defined as
\begin{align}
\label{dno-g}
G(\eta)\phi(x, \eta(x,t), t) = ((\nabla \phi)\cdot \widehat{\nu})(x,\eta(x)) \sqrt{1+|\partial_{x}\eta|^{2}},
\end{align}
where $\widehat{\nu}$ is the exterior unit normal used in \eqref{kinematic-condition-1-nu1}. For the upper fluid, the DNO is in matrix form due to the contribution from the traces of the velocity potential $\phi_1$ on the boundaries $\eta$ and $\eta_1$, and is defined as 
\begin{align}
\label{dno-gij}
\begin{pmatrix} G_{11}(\eta,\eta_{1}) & G_{12}(\eta,\eta_{1}) \\
G_{21}(\eta,\eta_{1}) & G_{22}(\eta,\eta_{1}) \end{pmatrix}
\begin{pmatrix} \phi_{1}(x, \eta(x,t)) \\ \phi_{1}(x, h_1+\eta_1(x,t)) \end{pmatrix}
&= \begin{pmatrix} -  \sqrt{1+|\partial_{x}\eta|^{2}}(\nabla \phi_{1} \cdot \hat{\nu})|_{y=\eta}  \\   \sqrt{1+|\partial_{x}\eta_{1}|^{2}}(\nabla \phi_{1} \cdot \hat{\nu_{1}})|_{ h_{1}+\eta_{1}} \end{pmatrix}.
\end{align}
The Hamiltonian for the free interface and free surface problem in \eqref{HamiltonianSum} becomes
\begin{equation}
\label{Hamiltonian-DNO-form}
\begin{aligned}
H &= \frac{1}{2} \int_{\mathbb{R}} \begin{pmatrix} \xi & \xi_{1} \end{pmatrix} \begin{pmatrix} G_{11}B^{-1}G & -GB^{-1}G_{12} \\ -G_{21}B^{-1}G & \rho^{-1} G_{22} - \rho \rho_{1}^{-1}G_{21}B^{-1}G_{12} \end{pmatrix} \begin{pmatrix} \xi \\ \xi_{1} \end{pmatrix}\, dx\\
&\quad + \frac{1}{2} \int_{\mathbb{R}} g(\rho-\rho_{1})\eta^{2}\, dx + \frac{1}{2} \int_{\mathbb{R}} g \rho_{1} \left ((h_{1}+\eta_{1})^{2} - h_{1}^{2} \right )\, dx,
\end{aligned}
\end{equation}
where 
\begin{equation}
\label{B-definition}
B(\eta, \eta_1) = \rho G_{11} (\eta, \eta_1) + \rho_1 G (\eta).    
\end{equation}


We recall  the expressions for the Taylor expansions of DNO operators in \eqref{dno-g}-\eqref{dno-gij} in powers
of the elevations $\eta$ and $\eta_1$. From Craig and Sulem \cite{CS93}, Craig, Guyenne and Kalisch \cite{CGK05} and Craig, Guyenne and Sulem \cite{CGSNatHazards11}, it is known that one can write $G(\eta)$ and $G_{ij}(\eta,\eta_{1})$ for the lower and upper fluid domains $S(\eta)$ and $S_{1}(\eta; \eta_{1})$, respectively, as 
\begin{equation}
\label{DNO-expansion}
\begin{cases}
G(\eta) = G^{(0)}(\eta) + G^{(1)}(\eta) + \mathcal{O}(\eta^{2}),\\
G_{ij}(\eta,\eta_{1}) = G_{ij}^{(0)}(\eta,\eta_{1}) + G_{ij}^{(1,0)}(\eta,\eta_{1}) + G_{ij}^{(0,1)}(\eta,\eta_{1}) + \mathcal{O}(|(\eta,\eta_{1})|^{2}), ~   i,j \in \{0,1\},
\end{cases}
\end{equation}
where the  $G^{(m)}(\eta)$ are homogeneous of degree $m$ in $\eta$, and  the $G^{(m_0,m_1)}_{ij}$ are homogeneous of degrees $m_0$ in $\eta$ and $m_1$ in $\eta_1$. 
The series in \eqref{DNO-expansion} are known to converge for sufficiently small $\eta$ and $\eta_1$ \cite{FN08}, and  their homogeneous terms can be found explicitly according to recursive relations given in \cite{CS93} and \cite{CGSNatHazards11}. 
Denoting $D = -i \partial_x$, we have
\begin{equation}
\label{DNO-1-definition}
\begin{cases}
G^{(0)}(\eta) = |D|, \\
G^{(1)}(\eta) = D \eta D - |D| \eta |D|,
\end{cases}
\end{equation}
where $|D|$ is a pseudo-differential operator acting as a Fourier multiplier $|k|$. Alternatively, $|D|$ can be expressed using the Hilbert transform $\H$ defined by
\begin{equation}
\label{hilbert-transform-defn}
\H h(x) := \text{p.v.} \frac{1}{\pi} \int\limits_{\mathbb{R}} \frac{h(y)}{x-y} dy.
\end{equation}
It is known that $\H = -i~ \sgn(D_x)$ and 
the relation $|D| = \partial_x \H$ holds.

For $G_{ij}$ we have 
\begin{equation}
\begin{pmatrix} G_{11}^{(0)} & G_{12}^{(0)} \\
G_{21}^{(0)} & G_{22}^{(0)} \end{pmatrix} =
\begin{pmatrix} 
D\coth(h_{1}D) & -D\mathrm{csch}(h_{1}D) \\
-D\mathrm{csch}(h_{1}D) & D\coth(h_{1}D) \end{pmatrix},
\end{equation}

\begin{equation}
\label{DNO-10-definition}
\begin{pmatrix} G_{11}^{(10)} & G_{12}^{(10)} \\
G_{21}^{(10)} & G_{22}^{(10)} \end{pmatrix} =
\begin{pmatrix} D\coth(h_{1}D)\eta D\coth(h_{1}D)-D\eta D & -D\coth(h_{1}D)\eta D~\mathrm{csch}(h_{1}D) \\
- D~\mathrm{csch}(h_{1}D)\eta D\coth(h_{1}D) & D~\mathrm{csch}(h_{1}D)\eta~ \mathrm{csch}(h_{1}D) \end{pmatrix},
\end{equation}

\begin{equation}
\label{DNO-01-definition}
\begin{pmatrix} G_{11}^{(01)} & G_{12}^{(01)} \\
G_{21}^{(01)} & G_{22}^{(01)} \end{pmatrix} = 
\begin{pmatrix} -D~\mathrm{csch}(h_{1}D)\eta_{1} D~\mathrm{csch}(h_{1}D) & D~\mathrm{csch}(h_{1}D)\eta_{1} D\coth(h_{1}D) \\
D\coth(h_{1}D)\eta_{1} D~\mathrm{csch}(h_{1}D) & -D\coth(h_{1}D)\eta_{1} D\coth(h_{1}D) + D\eta_{1}D \end{pmatrix}.\qquad \qquad
\end{equation}
From the explicit formulas above, we also find the expansion of the operator $B$ defined in \eqref{B-definition} and 
appearing in \eqref{Hamiltonian-DNO-form}:
\begin{equation}
\label{B-expansion}
\begin{aligned}
B & = \left(\rho G_{11}^{(0)} + \rho_1 G^{(0)}\right)
+ 
\left(\rho G_{11}^{(10)} + \rho G_{11}^{(01)} + \rho_{1} G^{(1)} \right) 
+ \mathcal{O}(|(\eta,\eta_{1})|^{2}) \\
& = : B_0 + B^{(1)} + \mathcal{O}(|(\eta,\eta_{1})|^{2}). 
\end{aligned}
\end{equation}

\section{Linear analysis}

We derive the linearized equations for the Hamiltonian system \eqref{Hamiltonian-system-1}. To do so, we extract the quadratic part, $H^{(2)}$, of the Hamiltonian $H$ in \eqref{Hamiltonian-DNO-form}.
We also introduce a canonical transformation which allows to diagonalize $H^{(2)}$.

\subsection{Quadratic Hamiltonian}


To extract $H^{(2)}$ from $H$ in canonical variables, we use the leading terms in the Taylor expansions of the operators given in \eqref{DNO-expansion} and \eqref{B-expansion}. 
We obtain
\begin{equation}
\label{H2-expression-1}
\begin{aligned}
H^{(2)} = &\frac{1}{2} \int_{\R} \bigg [\xi \frac{|D|\coth(h_{1}D)}{\rho \coth(h_{1}D) + \rho_{1}\mathrm{sgn}(D)} \xi - 2\xi \frac{|D|\mathrm{csch}(h_{1}D)}{\rho \coth(h_{1}D) + \rho_{1}\mathrm{sgn}(D)}\xi_{1} \\
&\qquad + \xi_{1} \frac{|D|\coth(h_{1}D) + \frac{\rho}{\rho_{1}}D}{\rho \coth(h_{1}D) + \rho_{1}\mathrm{sgn}(D)}\xi_{1} + g(\rho-\rho_{1})\eta^{2}\, dx + g\rho_{1}\eta_{1}^{2} \bigg ] \, dx.
\end{aligned}    
\end{equation}
Then, from \eqref{Hamiltonian-system-1}, the linearized equations of motion can be written as follows:
\begin{align*}
\partial_{t} \begin{pmatrix} \eta \\ \xi \\ \eta_{1} \\ \xi_{1} \end{pmatrix} 
= \begin{pmatrix} 0 & \frac{|D|\coth(h_{1}D)}{\rho \coth(h_{1}D) + \rho_{1}\mathrm{sgn}(D)} & 0 & -\frac{|D|\mathrm{csch}(h_{1}D)}{\rho \coth(h_{1}D) + \rho_{1}\mathrm{sgn}(D)} \\ -g(\rho - \rho_1) & 0 & 0 & 0 \\ 0 & -\frac{|D|\mathrm{csch}(h_{1}D)}{\rho \coth(h_{1}D) + \rho_{1}\mathrm{sgn}(D)} & 0 & \frac{|D|\coth(h_{1}D) + \frac{\rho}{\rho_{1}}D}{\rho \coth(h_{1}D) + \rho_{1}\mathrm{sgn}(D)} \\ 0 & 0 & -g\rho & 0 \end{pmatrix} 
\begin{pmatrix} \eta \\ \xi \\ \eta_{1} \\ \xi_{1} \end{pmatrix}
\end{align*}
The dispersion relation for this system  is 
\begin{align} \label{dispersion-relation-equation}
\omega^{4} - g\rho |k| \frac{1+\coth(h_{1}|k|)}{\rho \coth(h_{1}|k|) + \rho_{1}} \omega^{2} + g^{2}(\rho-\rho_{1})k^{2} \frac{1}{\rho \coth(h_{1}|k|) + \rho_{1}} = 0
\end{align}
with roots
\begin{equation}
\label{dispersion-relations}
\omega^{2}(k) = \frac{g (\rho-\rho_1) |k|}{\rho \coth(h_{1}|k|) +\rho_1 } \, ; \quad \omega_{1}^{2}(k) = g|k|,
\end{equation}
where $\omega^{2}(k)$ is associated with 
the interface, while  $\omega_{1}^{2}(D)$ with 
the surface wave. Note that $\omega^2(k)$ is the limiting value, as $h\to \infty$, of 
\begin{equation}
\label{dispersion-relation-finite-depth}
\begin{aligned}
&\omega^2(k) =  \frac{1}{2}g\rho k \frac{1+\tanh(hk) \coth (h_1k)}{\rho \coth(h_1k) + \rho_1 \tanh(hk)} \\
&- \frac{1}{2} gk \frac{\rho^2(1- \tanh(hk) \coth(h_1k))^2 + 4\rho \rho_1 \tanh(hk) (\coth(h_1k)-\tanh(hk)) + 4\rho_1^2 \tanh(hk)^2}{\rho \coth(h_1k) + \rho_1 \tanh(hk)}, \qquad
\end{aligned} 
\end{equation}
which corresponds to the dispersion relation in finite depth.

\subsection{Normal mode decomposition}
As in \cite{CGS12}, the surface and interface waves are coupled at first order in the Hamiltonian. To decouple the waves, we perform a normal mode decomposition using the canonical transformation
\begin{equation}
\label{canonical-transformation-1}
\begin{pmatrix}
\mu \\ \zeta\\ \mu_1 \\ \zeta_1
\end{pmatrix} = \begin{pmatrix}
a^- \sqrt{g (\rho - \rho_1)} & 0 & b^- \sqrt{g\rho_1} & 0 \\
0 & \frac{a^-}{\sqrt{g (\rho - \rho_1)}} & 0 & \frac{b^-}{\sqrt{g\rho_1}}\\
a^+ \sqrt{g (\rho - \rho_1)} & 0 & b^+ \sqrt{g\rho_1} & 0 \\
0 & \frac{a^+}{\sqrt{g (\rho - \rho_1)}} & 0 & \frac{b^+}{\sqrt{g\rho_1}}
\end{pmatrix} \begin{pmatrix}
\eta \\ \xi \\ \eta_1 \\ \xi_1
\end{pmatrix},
\end{equation}
where $a^\pm, b^\pm $ are the  Fourier multipliers
\begin{equation}
\label{a-b-pm-definition}
\begin{aligned}
& a^{\pm}(D) = \big(2 + \frac{\theta^2}{2} \pm \frac{\theta}{2}\sqrt{4+\theta^2} \big)^{-1/2},\; 
 b^{\pm} (D) = \frac{a^{\pm}(D)}{2}~(\theta \pm \sqrt{4+\theta^2}),\nonumber \\
&\theta (D) = \frac{\C(D)- \A(D)}{\B(D)}
\end{aligned}
\end{equation}
and the coefficients appearing in \eqref{a-b-pm-definition} are 
\begin{equation}
\label{A-B-C-coeff}
\begin{aligned}
& \A = g (\rho-\rho_1) G^{(0)} B_0^{-1} G_{11}^{(0)} \\
& \B  = - g \sqrt{\rho_1 (\rho-\rho_1)} G^{(0)} B_0^{-1} G_{12}^{(0)} \\
& \C = g G^{(0)} B_0^{-1} ( \rho_1 G_{11}^{(0)} + \rho G^{(0)} ).
\end{aligned}
\end{equation} 
After this transformation, the quadratic part of the Hamiltonian \eqref{H2-expression-1} simplifies to
\begin{equation}
\label{H2-expression-2}
\begin{aligned}
H^{(2)} &= \frac{1}{2} \int_{\mathbb{R}} \big [\zeta \omega^{2}(D)\zeta + \mu^{2} + \zeta_{1}\omega_{1}^{2}(D)\zeta_{1} + \mu_{1}^{2} \big ]\, dx,
\end{aligned}
\end{equation}
where $\omega^2(D)$ and $\omega_1^2 (D)$ are the eigenvalues of the symmetric matrix $ \begin{pmatrix} \A(D) & \B(D) \\ \B(D) & \C(D) \end{pmatrix}$ and have Fourier representations given by the dispersion relations \eqref{dispersion-relations}.
\begin{equation}
\label{symplectic-form-mu-mu_1}
\begin{aligned}
\partial_{t} \begin{pmatrix} \mu \\ \zeta \\ \mu_{1} \\ \zeta_{1} \end{pmatrix}
&= \begin{pmatrix} 0 & 1 & 0 & 0 \\
-1 & 0 & 0 & 0 \\
0 & 0 & 0 & 1 \\
0 & 0 & -1 & 0
\end{pmatrix}
\begin{pmatrix} \delta_{\mu}H \\ \delta_{\zeta}H \\
\delta_{\mu_{1}}H \\ \delta_{\zeta_{1}}H \end{pmatrix}.
\end{aligned}
\end{equation}

\section{Cubic terms in the Hamiltonian} 
We now turn to the cubic terms $H^{(3)}$ in $H$. 
The potential energy $V$ in \eqref{potential} is quadratic and does not contribute any higher-order terms. Thus, $H^{(3)}$ reduces to the cubic part $K^{(3)}$ of the kinetic energy $K$, which identifies with the first line of \eqref{Hamiltonian-DNO-form}. 
We calculate $K^{(3)}$ in the original variables $(\eta,\xi, \eta_1, \xi_1)$.

\subsection{Cubic terms of the kinetic energy in the original variables}

We rewrite the kinetic energy as $K = \mathrm{I} - \mathrm{II} + \mathrm{III}$,
where
\begin{align}
\label{kinetic-energy-terms}
\begin{cases}
\mathrm{I} = \frac{1}{2} \int_{\mathbb{R}} \xi G_{11}B^{-1}G\xi \, dx \\ 
\mathrm{II} = \int_{\mathbb{R}} \xi G B^{-1}G_{12}\xi_{1} \, dx \\ 
\mathrm{III} = \frac{1}{2} \int_{\mathbb{R}} \xi_{1} \left (\rho_{1}^{-1}G_{22} - \rho \rho_{1}^{-1}G_{21}B^{-1}G_{12} \right )\xi_{1} \, dx,
\end{cases}
\end{align}
and  expand the operators appearing in $\mathrm{I},  \mathrm{II},  \mathrm{III} $ at first order in $\eta$ or $\eta_1$ using
\eqref{DNO-expansion}, \eqref{DNO-10-definition}, \eqref{DNO-01-definition} and \eqref{B-expansion} as well as identities $G_{11}^{(0)} = G_{22}^{(0)}$, $G_{12}^{(0)} = G_{21}^{(0)}$ and $(G_{11}^{(0)})^{2} - (G_{12}^{(0)})^{2} = (G^{(0)})^{2}$.
\begin{proposition} \label{proposition-all-terms}
The cubic order terms $\mathrm{I}^{(3)}, \mathrm{II}^{(3)}, \mathrm{III}^{(3)}$ in  $\mathrm{I},\mathrm{II}, \mathrm{III}$ of \eqref{kinetic-energy-terms} respectively, are given by:

\begin{equation}
\label{I-3-definition}
\begin{aligned}
\mathrm{I}^{(3)}
&= \frac{1}{2}\int_{\mathbb{R}}         \Big[ -\rho \eta ~(DB_0^{-1} G_{11}^{(0)}\xi)^2
-(\rho-\rho_1)\eta ~(G^{(0)}B_0^{-1} G_{11}^{(0)}\xi)^2  \\
& \quad + \rho_1 \eta  ~(DB_0^{-1} G^{(0)}\xi)^2
 - \rho_1 \eta_1 ~(G^{(0)}B_0^{-1}G_{12}^{(0)}\xi)^2\Big]\, dx;
\end{aligned}
\end{equation}
\medskip
\begin{equation}
\begin{aligned}
&\mathrm{II}^{(3)} = \int_{\mathbb{R}} \Big[        - \rho\eta (DB_0^{-1} G_{11}^{(0)} \xi) 
(DB_0^{-1} G_{12}^{(0)} \xi_1) - (\rho-\rho_1) \eta 
(G^{(0)}B_0^{-1} G_{11}^{(0)} \xi) 
( G^{(0)}B_0^{-1} G_{12}^{(0)} \xi_1) \qquad\\
& -\rho\eta (D B_0^{-1} G^{(0)} \xi) (DB_0^{-1} G_{12}^{(0)} \xi_1)    
- \eta_1 (G_{12}^{(0)} B_0^{-1} G^{(0)} \xi)
\left( G^{(0)} B_0^{-1} (\rho_1 G_{11}^{(0)} + \rho G^{(0)}) \xi_1 \right)\Big]\, dx;\qquad
\end{aligned}    
\end{equation}
\medskip
\begin{equation}
\begin{aligned}
&\mathrm{III}^{(3)} =  \frac{1}{2}\int_{\mathbb{R}} \Big[       - (\rho-\rho_1) \eta 
(G^{(0)}B_0^{-1}G_{12}^{(0)} \xi_1)^2
+ \frac{\rho}{\rho_1}(\rho-\rho_1)\eta
(DB_0^{-1}G_{12}^{(0)} \xi_1)^2
\\
&\quad -\frac{1}{\rho_1} \eta_1 \left( G^{(0)} B_0^{-1} (\rho_1 G_{11}^{(0)} + \rho G^{(0)}) \xi_1 \right) 
-\frac{1}{\rho_1}\eta_1 (D\xi_1)^2 \Big]\, dx.
\end{aligned}
\end{equation}
\end{proposition}
Regrouping the contributions obtained in the above proposition, we obtain:
\begin{proposition} \label{cubicH-summary}
The cubic terms $H^{(3)}$ of the Hamiltonian $H$ can be written as 
\begin{equation}
\begin{aligned} \label{cubic-Hamiltonian-1}
&H^{(3)} 
= \frac{1}{2} \int_{\mathbb{R}} \bigg [-(\rho-\rho_{1})\eta \left (G^{(0)}B_{0}^{-1}(G_{11}^{(0)}\xi - G_{12}^{(0)}\xi_{1}) \right )^{2} \\
&
 -\rho_{1}\eta_{1} \left (G_{12}^{(0)} B_{0}^{-1} G^{(0)}\xi - \frac{1}{\rho_{1}} G^{(0)} B_0^{-1} (\rho_1 G_{11}^{(0)} + \rho G^{(0)}) \xi_1
\right )^{2}  \\
&
 -\rho \eta \left (DB_{0}^{-1}(G_{11}^{(0)}\xi - G_{12}^{(0)}\xi_{1}) \right )^{2}  
 + \rho_{1} \eta \left (DB_{0}^{-1}G^{(0)}\xi + \frac{\rho}{\rho_{1}} D B_{0}^{-1} G_{12}^{(0)}\xi_{1} \right )^{2} - \frac{1}{\rho_{1}}\eta_{1}(D\xi_{1})^{2} \bigg ] dx.\qquad
\end{aligned}
\end{equation}
\end{proposition}
The proofs of Propositions \ref{proposition-all-terms}-\ref{cubicH-summary} are  a little lengthy but straightforward. The main steps are given in Appendix \ref{appendix-proof}. 

\subsection{Cubic terms of the kinetic energy in normal modes coordinates}
We now write the cubic   terms of the Hamiltonian in normal modes coordinates \eqref{canonical-transformation-1}.
\begin{proposition} \label{cubic in normal variables}
The cubic part of the Hamiltonian \eqref{cubic-Hamiltonian-1} in the normal-mode coordinates  \eqref{canonical-transformation-1} is given by
\begin{equation}
\label{cubic-Hamiltonian-mu-zeta-form}
\begin{aligned} 
H^{(3)} &= -\frac{\rho-\rho_{1}}{2\sqrt{g(\rho-\rho_{1})}}\int_{\mathbb{R}} (b^{+}\mu - b^{-}\mu_{1}) \left (\mathcal{A}_{1}\zeta - \mathcal{B}_{1}\zeta_{1} \right )^{2} \, dx \\
&\quad +\frac{\rho_{1}}{2\sqrt{g\rho_{1}}} \int_{\mathbb{R}} (a^{+}\mu - a^{-}\mu_{1}) \left (\mathcal{A}_{2}\zeta-\mathcal{B}_{2}\zeta_{1} \right )^{2} \, dx \\
&\quad -\frac{\rho}{2\sqrt{g(\rho-\rho_{1})}} \int_{\mathbb{R}} (b^{+}\mu-b^{-}\mu_{1}) \left (\mathcal{A}_{3}\zeta - \mathcal{B}_{3}\zeta_{1} \right )^{2}\, dx \\
&\quad + \frac{\rho_{1}}{2\sqrt{g(\rho-\rho_{1})}} \int_{\mathbb{R}} (b^{+}\mu-b^{-}\mu_{1}) \left (\mathcal{A}_{4}\zeta - \mathcal{B}_{4}\zeta_{1} \right )^{2}\, dx \\
&\quad + \frac{1}{2\rho_{1}\sqrt{g\rho_{1}}} \int_{\mathbb{R}} (a^{+}\mu - a^{-}\mu_{1}) \left (\mathcal{A}_{5}\zeta - \mathcal{B}_{5}\zeta_{1} \right )^{2} \, dx,\\
& = : R_1 + R_2 + R_3+R_4+R_5.
\end{aligned}    
\end{equation}
where $R_j$ stands for the integral expression in the row $j$ of \eqref{cubic-Hamiltonian-mu-zeta-form} and 
\begin{equation}
\label{mathcal-A-B-expressions}
\begin{aligned}
\mathcal{A}_{1} (D) &:= \frac{1}{\sqrt{g(\rho-\rho_1)}} (b^+ Q_a - a^+ Q_b), \quad \mathcal{B}_{1} (D) :=  \frac{1}{\sqrt{g(\rho-\rho_1)}} (b^- Q_a - a^- Q_b),\\
\mathcal{A}_{2} (D) &:= \frac{1}{\sqrt{g\rho_1}} (a^+ Q_c - b^+ Q_b),\quad
\mathcal{B}_{2} (D) := \frac{1}{\sqrt{g\rho_1}} (a^- Q_c - b^- Q_b)\\
\mathcal{A}_3 (D) &: = \sgn(D) \mathcal{A}_1 (D), \qquad \mathcal{B}_3 (D): = \sgn(D) \mathcal{B}_1 (D), \\
\mathcal{A}_{4} (D) & := b^{+}\sqrt{g(\rho-\rho_{1})}DB_{0}^{-1}G^{(0)} + \frac{\rho~ \sgn(D)}{\rho_{1} \sqrt{g(\rho-\rho_1)}}a^{+}Q_b, \\
\mathcal{B}_{4} (D) &:= b^{-}\sqrt{g(\rho-\rho_{1})}DB_{0}^{-1}G^{(0)} + \frac{\rho~ \sgn(D)}{\rho_{1} \sqrt{g(\rho-\rho_1)}}a^{-}Q_b,\\
\mathcal{A}_{5} (D) &= -\sqrt{g\rho_{1}} D a^{+}, \qquad
\mathcal{B}_{5} (D) = -\sqrt{g\rho_{1}} D a^{-} ~.
\end{aligned}
\end{equation}
\end{proposition}

\noindent The proof consists of rewriting each term of \eqref{cubic-Hamiltonian-1} in terms of new variables \eqref{canonical-transformation-1}, using the inverse relations
\begin{align} \label{inverse transformation}
\begin{cases}
\eta &= \frac{1}{\sqrt{g(\rho-\rho_{1})}} \left( b^+ \mu - b^- \mu_{1} \right) \\
\eta_{1} &= -\frac{1}{\sqrt{g\rho_{1}}}\left(a^+ \mu - a^- \mu_{1} \right)
\end{cases}
\quad \text{and} \quad 
\begin{cases}
\xi &= \sqrt{g(\rho-\rho_{1})} \left( b^{+} \zeta - b^{-} \zeta_{1} \right) \\
\xi_{1} &= -\sqrt{g\rho_{1}} \left( a^{+} \zeta - a^{-} \zeta_{1} \right)
\end{cases}
\end{align}
and grouping the terms.

\section{Benjamin-Ono scaling and modulational Ansatz}

We now introduce the scaling regime under consideration. We assume the typical wavelength $\lambda$ of the internal wave is large
when it is compared to the typical depth of the upper layer $h_1$,  and its amplitude  $a$  of a typical wave is small when compared to $h_1$, with the relation
$$\varepsilon =  \frac{h_1}{\lambda} = \frac{a}{h_1} \ll 1.$$
The following scalings are thus introduced
\begin{equation}
\label{long-wave-scaling}
X = \varepsilon x, \qquad \mu(X) = \varepsilon \widetilde{\mu}(X), \qquad \zeta (x) = \widetilde{\zeta}(X).
\end{equation}
On the other hand, the surface modes are in the form of near-monochromatic waves with carrier wavenumber $k_{0} >0$ and amplitude $a_{1}$ satisfying  $\varepsilon_{1} =  k_{0}a_{1} \ll 1.$
We thus define 
\begin{equation}
\label{mono-waves-scaling}
\begin{aligned}
\mu_{1}(x,t) &= \frac{\varepsilon_{1}}{\sqrt{2}}\omega_{1}^{1/2}(D_{x})(\q(X,t)e^{ik_{0}x} + \overline{\q}(X,t)e^{-ik_{0}x})  \\
\zeta_{1}(x,t) &= \frac{\varepsilon_{1}}{\sqrt{2}i}\omega_{1}^{-1/2}(D_{x})(\q(X,t)e^{ik_{0}x} - \overline{\q}(X,t)e^{-ik_{0}x}),
\end{aligned}
\end{equation}
where the weights $\omega_{1}^{1/2}(D)$ have been introduced to diagonalize the corresponding quadratic part of the Hamiltonian.

In addition, we assume the relation $\varepsilon_{1} = \varepsilon^{1+\delta}$ with $\delta \in \left(0,\frac{1}{2}\right)$  between the small parameters $\varepsilon_1$ and $\varepsilon$ to emphasize that  the internal wave amplitude is larger than that of the surface wave. The restriction $\delta < \frac{1}{2}$ relates to the ordering of different small terms in the expansion, see Remark \ref{remark-delta-values}.


To incorporate the Benjamin-Ono scaling and the modulational Ansatz into the Hamiltonian, we need two key elements of asymptotic analysis. Lemma \ref{lemma-fourier-action} addresses the multi-scale character of the problem of fast oscillations versus long-wave scaling, by describing the action of a Fourier multiplier on a multiple-scale function \cite{CSSNonl92}. 
 On the other hand, Lemma \ref{lemma-scale-separation} is a scale-separation lemma (in a simplified form of Lemma 3.2 of \cite{CGNS05})  that expresses that fast oscillations essentially homogenize and do not contribute to the Hamiltonian. In our case, $g(x) = e^{i\alpha x}$ is continuous and periodic with respect to the translation lattice $\frac{2\pi}{\alpha}\mathbb{Z}$. 
\begin{lemma}
\label{lemma-fourier-action}
Let $m(D_x)$ be a Fourier multiplier, then for sufficiently smooth $f$ and sufficiently small $\varepsilon$ we have the following asymptotic expansion
\begin{equation}
\label{fourier-multiplier-action-k0}
\begin{aligned}
m(D_x) \left( e^{ik_0x} f(X) \right) & = e^{ik_0x} m(k_0 + \varepsilon D_X) f(X) \\
& = e^{ik_0x} \left( m(k_0) + \varepsilon \partial_k m(k_0) D_X f(X) + \frac{\varepsilon^2}{2} \partial_k^2 m(k_0) D_X^2 f(X) + ... \right),
\end{aligned}
\end{equation}
where $D_x = \varepsilon D_X$.
Letting $k_0=0$, the expansion becomes
\begin{equation}
\label{fourier-multiplier-action-0}
\begin{aligned}
m(D_x) f(X) & = m(\varepsilon D_X) f(X) \\
& = \left( m(0) + \varepsilon \partial_k m(0) D_X f(X) + \frac{\varepsilon^2}{2} \partial_k^2 m(0) D_X^2 f(X) + ... \right).
\end{aligned}
\end{equation}
\end{lemma}

\begin{lemma}
\label{lemma-scale-separation}
Let $f$ be a real-valued function of Schwartz class, $\alpha \in \mathbb{R}$ be a nonzero
constant and $\varepsilon$ be sufficiently small. Then, for all natural $N$,
\begin{equation}
\label{scale-separation}
\int_\mathbb{R} e^{i\alpha x} f(X) dx = \mathcal{O} (\varepsilon^N).
\end{equation}
\end{lemma}

\subsection{The quadratic Hamiltonian under the scaling regime}
Inserting the above scalings into the quadratic Hamiltonian \eqref{H2-expression-2}, we get 
\begin{equation}
\label{H2-expression-modulational-initial}
\begin{aligned}
H^{(2)} &= \frac{1}{2} \int_{\R} \left( \varepsilon^{-1} \widetilde{\zeta}~ \omega^2(\varepsilon D_X) \widetilde{\zeta}   + \varepsilon \widetilde{\mu}^{2}  + \varepsilon^{-1} \varepsilon_1^2 \overline{\q}~ \omega_1 (k_0 + \varepsilon D_X) \q \right) ~dX,
\end{aligned}
\end{equation}
where $D_x = \varepsilon D_X$. The expansion for $\omega_1(k_0+\varepsilon D_X)$ follows \eqref{fourier-multiplier-action-k0} since $\omega_1$ in \eqref{dispersion-relations} is smooth around $k_0$. In contrast, $\omega^2$ is not sufficiently smooth near $k=0$, so its expansion is obtained from the expansion of \eqref{dispersion-relation-finite-depth} in the deep water limit $h\to \infty$: 
\begin{equation}
\label{omega-expansion}
\begin{aligned}
\omega^2(\varepsilon D_X) = \varepsilon^{2} 
\Omega_0
D_{X}^{2} + \varepsilon^{3}
\Omega_1 D_{X}^{2}|D_{X}| + \varepsilon^{4} 
\Omega_2 D_{X}^{4} + \mathcal{O}(\varepsilon^5),
\end{aligned}
\end{equation}
with
\begin{equation}
\label{Omega-defn}
\begin{aligned}
&
\Omega_0= \frac{gh_1(\rho-\rho_1)}{\rho}, \quad
\Omega_1 = - \frac{g(\rho-\rho_1)\rho_1 h_1^2}{\rho^2}, \quad 
\Omega_2 = \frac{g(\rho-\rho_1) h_1^3}{\rho} \left(\frac{\rho_1^2}{\rho^2} - \frac{1}{3} \right).
\end{aligned}
\end{equation}
The quadratic Hamiltonian in \eqref{H2-expression-modulational-initial} becomes
\begin{equation}
\label{H2-expression-modulational}
\begin{aligned}
H^{(2)} &= \int_{\R} \bigg [-\varepsilon \frac{\Omega_0}{2} (D_{X}\widetilde{\zeta})^{2} + \varepsilon^{2} \frac{\Omega_1}{2} \widetilde{\zeta} (D_{X}^{2} |D_X| \widetilde{\zeta}) \\
&\qquad \qquad \qquad  + \varepsilon^{3} \frac{\Omega_2}{2} \widetilde{\zeta}(D_{X}^{4}\widetilde{\zeta}) + \varepsilon \widetilde{\mu}^{2} + \frac{\varepsilon_{1}^{2}}{\varepsilon} \omega_{1}(k_{0}) |\q|^{2} \\
&\qquad \qquad \qquad  + \varepsilon_{1}^{2}\omega_{1}'(k_{0})\overline{\q}(D_{X}\q) + \varepsilon \varepsilon_{1}^{2} \frac{\omega_{1}''(k_{0})}{2} \overline{\q}(D_{X}^{2}\q) \bigg ]\, dX + \mathcal{O}(\varepsilon^{4}).
\end{aligned}
\end{equation}
The transformations \eqref{long-wave-scaling}-\eqref{mono-waves-scaling} retain the standard symplectic form of \eqref{symplectic-form-mu-mu_1}, and the corresponding equation of motion in new variables can be written as 
\begin{equation}
\label{symplectic-form-tilde-mu-v_1}
\begin{aligned}
\partial_{t} \begin{pmatrix} \widetilde\mu \\ \widetilde\zeta \\ \q \\ \overline{\q} \end{pmatrix}
&= \begin{pmatrix} 0 & 1 & 0 & 0 \\
-1 & 0 & 0 & 0 \\
0 & 0 & 0 & -i\varepsilon \varepsilon_1^{-2} \\
0 & 0 & i\varepsilon \varepsilon_1^{-2} & 0
\end{pmatrix}
\begin{pmatrix} \delta_{\widetilde\mu}H \\ \delta_{\widetilde\zeta}H \\
\delta_{\q} H \\ \delta_{\overline{\q}} H \end{pmatrix}.
\end{aligned}
\end{equation}

\subsection{The cubic Hamiltonian under the scaling regime}

For the cubic Hamiltonian \eqref{cubic-Hamiltonian-mu-zeta-form} we show that, after the transformations \eqref{long-wave-scaling}--\eqref{mono-waves-scaling}, some of its terms become of higher order and do not contribute at the order of our approximation. 

Expanding the brackets inside the integral of $R_1$ in \eqref{cubic-Hamiltonian-mu-zeta-form}, 
the terms involving $(b^+ \mu) (\mathcal{A}_1 \zeta) (\mathcal{B}_1 \zeta_1)$, $(b^- \mu_1) (\mathcal{A}_1 \zeta)^2$ and $(b^- \mu_1) (\mathcal{B}_1 \zeta_1)^2$ 
can be estimated by the scale separation lemma as follows 
\begin{equation}
\label{scale-separ-lemma-application}
\begin{aligned}
\int_\mathbb{R} (b^- \mu_1) (\mathcal{A}_1 \zeta)^2 dx
&= \int_{\R} \left (\frac{\varepsilon_{1}}{\sqrt{2}} (b^- \omega_{1}^{1/2})(D_{x})\big (\q(X)e^{ik_{0}x} + c.c. \big ) \right )(\mathcal{A}_1(D_{x})\zeta)^{2} \, dx \\
&= \frac{\varepsilon_{1}}{\sqrt{2}} \int_{\R} e^{ik_{0}x} \left ((b^-\omega_{1}^{1/2})(k_{0}+\varepsilon D_{X})\q(X) \right )(\mathcal{A}_1 (\varepsilon D_{X})\widetilde{\zeta})^{2} \, dx + \mathrm{c.c.} 
\end{aligned}
\end{equation}
Inside the integral we have a product of a fast oscillating function $e^{ik_0 x}$ with a slowly modulated smooth function of $X$. Applying the scale separation Lemma \ref{lemma-scale-separation} to 
this integral, the integral is of order $\mathcal{O}(\varepsilon_1 \varepsilon^N)$
for sufficiently large $N$. 
The other two terms in $R_1$ as well as similar terms in $R_2, ..., R_5$ are treated in the same fashion, the details of which are in \cite{CPAK23}. The remaining terms in the cubic Hamiltonian are
\begin{equation}
\label{cubic-Hamiltonian-leading-terms}
\begin{aligned}
&H^{(3)} = -\frac{\rho-\rho_{1}}{2\sqrt{g(\rho-\rho_{1})}}\int_{\R} \left [(b^{+}\mu)(\mathcal{A}_{1}\zeta)^{2} + (b^{+}\mu)(\mathcal{B}_{1}\zeta_{1})^{2} + 2(b^{-}\mu_{1})(\mathcal{A}_{1}\zeta)(\mathcal{B}_{1}\zeta_{1}) \right ] dx \\
&\qquad +\frac{\rho_{1}}{2\sqrt{g\rho_{1}}} \int_{\R} \left [(a^{+}\mu)(\mathcal{A}_{2}\zeta)^{2} + (a^{+}\mu)(\mathcal{B}_{2}\zeta_{1})^{2} + 2(a^{-}\mu_{1})(\mathcal{A}_{2}\zeta)(\mathcal{B}_{2}\zeta_{1}) \right ] dx\\
&\qquad -\frac{\rho}{2\sqrt{g(\rho-\rho_{1})}} \int_{\R} \left [(b^{+}\mu)(\mathcal{A}_{3}\zeta)^{2} + (b^{+}\mu)(\mathcal{B}_{3}\zeta_{1})^{2} + 2(b^{-}\mu_{1})(\mathcal{A}_{3}\zeta)(\mathcal{B}_{3}\zeta_{1}) \right ] dx \\
&\qquad + \frac{\rho_{1}}{2\sqrt{g(\rho-\rho_{1})}} \int_{\R} \left[(b^{+}\mu)(\mathcal{A}_{4}\zeta)^{2} + (b^{+}\mu)(\mathcal{B}_{4}\zeta_{1})^{2} + 2(b^{-}\mu_{1})(\mathcal{A}_{4}\zeta)(\mathcal{B}_{4}\zeta_{1}) \right ] dx \\
& + \frac{1}{2\rho_{1}\sqrt{g\rho_{1}}} \int_{\R} \left [(a^{+}\mu )(\mathcal{A}_{5}\zeta)^{2} + (a^{+}\mu)(\mathcal{B}_{5}\zeta_{1})^{2} + 2(a^{-}\mu_{1})(\mathcal{A}_{5}\zeta)(\mathcal{B}_{5}\zeta_{1}) \right ] dx + \mathcal{O} (\varepsilon_1 \varepsilon^N).
\end{aligned}
\end{equation}
\begin{proposition}
\label{proposition-H3-kappa}
The cubic part of the Hamiltonian $H^{(3)}$ can be written as
\begin{equation}
\label{H3-in-tilde-variables}
\begin{aligned}
H^{(3)} = & \int_{\R} \bigg[\varepsilon^{2} \kappa \widetilde{\mu} (D_{X}\widetilde{\zeta})^{2} + \varepsilon_{1}^{2}(\kappa_{1}\widetilde{\mu} + \kappa_{2}\partial_{X}\widetilde{\zeta})|\q|^{2} + \varepsilon^{3}\kappa_{3} \widetilde{\mu}(|D_{X}|D_{X}\widetilde{\zeta})(D_{X}\widetilde{\zeta}) \\
&\qquad + \varepsilon \varepsilon_{1}^{2}(\kappa_{4}\widetilde{\mu} + \kappa_{5} \partial_{X}\widetilde{\zeta} ) \left(\q \overline{D_{X}\q} + \overline{\q} D_{X}\q \right) \\
&\qquad + \varepsilon \varepsilon_{1}^{2}\big (\kappa_{6} |D_{X}|\widetilde{\mu} + \kappa_{7} |D_{X}|\partial_{X}\widetilde{\zeta} \big )|\q|^{2} + \varepsilon^{3}\kappa_{8}(|D_{X}|\widetilde{\mu})(D_{X}\widetilde{\zeta})^{2} \bigg ]\, dX  + \mathcal{O}(\varepsilon^{4}),
\end{aligned}
\end{equation}
where the coefficients $\kappa$ and $\kappa_{j}$ only depend on physical parameters $g$, $h_{1}$, $\rho$ and $\rho_{1}$. Their expressions are given in Appendix \ref{appendix-coefficients}.
\end{proposition}

\begin{proof}
Substitution of \eqref{long-wave-scaling}--\eqref{mono-waves-scaling} into $H^{(3)}$ leads to terms of different kind. For example, we sketch the calculations for the terms appearing in line $4$ of \eqref{cubic-Hamiltonian-mu-zeta-form}. 
\newline \\
For the first term $\int (b^+ \mu) (\mathcal{A}_4 \zeta)dx$, we write
\begin{equation*}
b^+(D_x) \mu(x) = \varepsilon b^+ (\varepsilon D_X) \widetilde \mu (X).
\end{equation*}
with
\begin{equation}
\label{b-plus-expansion}
\begin{aligned}
b^+(\varepsilon D_X) & = \sqrt{\frac{\rho_1}{\rho}} + \varepsilon \frac{\rho-\rho_1}{\rho} \sqrt{\frac{\rho_1}{\rho}} h_1 |D_X| + \mathcal{O}(\varepsilon^{2}) \\ 
&=: (b^{+})^{(0)} + \varepsilon (b^{+})^{(1)}|D_{X}| + \mathcal{O}(\varepsilon^{2}).
\end{aligned}
\end{equation}
Similar steps show that, since every term of $\mathcal{A}_4$ contains the element $D_x$, its expansion is 
\begin{equation}
\label{A4-expansion}
\begin{aligned}
\mathcal{A}_4(\varepsilon D_X) & =  \varepsilon \sqrt{\frac{g(\rho-\rho_1)}{\rho_1 \rho}} \left(1-\varepsilon \frac{\rho_1}{\rho} h_1 |D_X|\right) D_X + \mathcal{O}(\varepsilon^3) \\
& := \varepsilon \mathcal{A}_4^{(0)} D_X + \varepsilon^2 \mathcal{A}_4^{(1)}|D_{X}| D_X + \mathcal{O}(\varepsilon^3)
\end{aligned}
\end{equation}
with $\mathcal{A}_4^{(0)}$ and $\mathcal{A}_4^{(1)}$ representing constants.
As a result,
\begin{equation}
\label{b-mu-A4-zeta-term}
\begin{aligned}
\left( b^+ \mu \right) \left( \mathcal{A}_4 \zeta \right)^2 = &~ \varepsilon^3 (b^+)^{(0)} \left( \mathcal{A}_4^{(0)} \right)^2 \widetilde \mu \left( D_X \widetilde \zeta \right)^2 + \varepsilon^4 (b^+)^{(1)} \left( \mathcal{A}_4^{(0)} \right)^2 \left( |D_X| \widetilde \mu \right) \left( D_X \widetilde \zeta \right)^2 \\
& + 2 \varepsilon^4 (b^+)^{(0)}  \mathcal{A}_4^{(0)} \mathcal{A}_4^{(1)} 
\widetilde \mu \left( D_X \widetilde \zeta \right) \left( |D_X| D_X \widetilde \zeta \right) + \mathcal{O} (\varepsilon^5). 
\end{aligned}
\end{equation}
The coefficients of the first, second and third terms of \eqref{b-mu-A4-zeta-term} contribute to $\kappa$, $\kappa_8$ and $\kappa_3$, respectively. 

For the second term $\int (b^+ \mu) (\mathcal{B}_4 \zeta_1)^2 dx$ 
$\left( \mathcal{B}_4 \zeta_1 \right)^2$ we apply \eqref{fourier-multiplier-action-k0} to get
\begin{equation*}
\begin{aligned}
\left( \mathcal{B}_4 \zeta_1 \right)^2 = - \left( \mathcal{B}_4^2 \omega_1^{-1} \right)(k_0) |\q|^2 - \frac{\varepsilon}{2} \left( \mathcal{B}_4^2 \omega_1^{-1} \right)' (k_0) \left(\q \overline{D_{X}\q} + \overline{\q} D_{X}\q \right) + ...,
\end{aligned}
\end{equation*}
where $\mathcal{B}_4 (k)$ is the Fourier symbol of $\mathcal{B}_4 (D_x)$,
and we omit the terms with fast exponential functions $e^{ik_0x}$ since they disappear under the integral due to the scale-separation Lemma \ref{lemma-scale-separation}. Then, we have
\begin{equation*}
\begin{aligned}
(b^{+}\mu)(\mathcal{B}_{4}\zeta_{1})^{2} = &  -\varepsilon\varepsilon_{1}^{2} (b^{+})^{(0)} (\mathcal{B}_{4}^{2} \omega_{1}^{-1})(k_{0}) ~\widetilde{\mu} |\q|^{2} - \varepsilon^{2} \varepsilon_{1}^{2} (b^{+})^{(1)}(\mathcal{B}_{4}^{2}\omega_{1}^{-1})(k_{0})\left( |D_{X}| \widetilde{\mu} \right ) |\q|^{2}  \\
& - \frac{\varepsilon \varepsilon_{1}^{2}}{2} (b^{+})^{(0)} (\mathcal{B}_{4}^{2}\omega_1^{-1})'(k_{0}) \widetilde{\mu}\left( 
\q \overline{D_{X}\q} + \overline{\q} D_{X}\q \right) + \mathcal{O}(\varepsilon^{3}\varepsilon_{1}^{2}),
\end{aligned}
\end{equation*}
where the coefficients of the first, second and third terms contribute into $\kappa_1$, $\kappa_6$ and $\kappa_4$, respectively.
We similarly write the expansion for $(b^{-}\mu_{1})(\mathcal{A}_{4}\zeta)(\mathcal{B}_{4}\zeta_{1})$, which contributes into $\kappa_2$, $\kappa_5$ and $\kappa_7$. 
\end{proof}

\section{Derivation of the coupled Benjamin-Ono -- Schr\"{o}dinger system.}

In this section we derive a Benjamin-Ono -- Schr\"{o}dinger system for the Hamiltonian system \eqref{symplectic-form-tilde-mu-v_1} following the approach used in \cite{CGSNatHazards11}.

\subsection{Resonance condition}
The expansions \eqref{H2-expression-modulational} and \eqref{H3-in-tilde-variables} provide the expressions for the leading terms of the Hamiltonian $H$ in variables $(\widetilde \mu, \widetilde \zeta)$. The Hamiltonian can be further simplified using a moving frame of reference or, equivalently, combining it with the conserved momentum \eqref{momentum}. 
In the $(\zeta, \eta)$ variables, the momentum becomes 
\begin{equation*}
I = -\int_{\R} (\zeta \partial_{x}\mu + \zeta_{1} \partial_{x}\mu_{1})\, dx,
\end{equation*}
which after expansion under the Benjamin-Ono scaling and modulational Ansatz implies 
\begin{equation}
\label{momentum-in-new-variables}
\begin{aligned}
I = -\int_{\R} i\varepsilon \widetilde{\zeta}(D_{X}\widetilde{\mu}) -\frac{\varepsilon_{1}^{2}}{\varepsilon}k_{0}|\q|^{2} -\frac{\varepsilon_{1}^{2}}{2} \big [(D_{X}\q)\overline{\q} + \q(\overline{D_{X}\q}) \big ] \, dX + \mathcal{O}(\varepsilon_{1}^{2}\varepsilon^{N}).
\end{aligned}
\end{equation}
Combining the Hamiltonian with the momentum \eqref{momentum-in-new-variables}, we have
\begin{equation}
\label{reduced-Hamiltonian-tilde-mu-q-2}
\begin{aligned}
&H- cI = H^{(2)}-cI + H^{(3)} + H^{(4)} + ... \\
& = \int_{\R} \bigg [\varepsilon \frac{\Omega_0}{2} \left( \frac{\widetilde{\mu}^{2}}{\Omega_0}\left (1 - \frac{c^{2}}{\Omega_0} \right ) - \left (D_{X}\widetilde{\zeta} + \frac{ic\widetilde{\mu}}{\Omega_0}\right )^{2} \right) \\
& + \frac{\varepsilon_{1}^{2}}{\varepsilon}(\omega_{1}(k_{0})-ck_{0}) |\q|^2+ \varepsilon^{2} \frac{\Omega_1}{2} \widetilde{\zeta} (D_{X}^{2}|D_{X}| \widetilde{\zeta}) + \varepsilon^{3} \frac{\Omega_2}{2} \widetilde{\zeta}(D_{X}^{4}\widetilde{\zeta}) \\
&+ \frac{\varepsilon_{1}^{2}}{2}(\omega_{1}'(k_{0})-c) \left(\overline{\q} D_{X}\q + \q \overline{D_{X}\q} \right) + \frac{\varepsilon \varepsilon_{1}^{2}}{2} \omega_{1}''(k_{0})\overline{\q}(D_{X}^{2}\q) \bigg ]\, dX + H^{(3)} + \mathcal{O}\left(\varepsilon^4 + \frac{\varepsilon_1^4}{\varepsilon}\right).
\end{aligned}
\end{equation}
We choose 
\begin{equation}
\label{c0}
c = c_0:= \sqrt{\Omega_0} = \sqrt{ g\left( 1-\frac{\rho_1}{\rho} \right) h_1 }, 
\end{equation}
that eliminates $\widetilde \mu^2$ term in $\widehat H$. We then select the wavenumber $k_{0}$ of the carrier wave
\begin{equation}
\label{k0}
k_{0} =\frac{\rho}{4h_{1}\left(\rho - \rho_1 \right)}
\end{equation}
such that the phase velocity of the internal wave is equal to the group velocity of the surface wave.
\begin{equation}
\omega_1'(k_{0}) = c_{0},
\end{equation}
This is referred to as a resonance condition between two waves.
The flow of the Hamiltonian $H-c_0 I$ at our order of approximation  conserves the generalized wave action
\begin{equation*}
M := \int_{\R} |\q|^{2}\, dX.
\end{equation*}
Subtracting the multiple of $M$ from \eqref{reduced-Hamiltonian-tilde-mu-q}, we further reduce the Hamiltonian to 
\begin{equation}
\label{reduced-Hamiltonian-tilde-mu-q}
\begin{aligned}
\widehat H & = H- cI - \frac{\varepsilon_{1}^{2}}{\varepsilon}(\omega_{1}(k_{0})-ck_{0}) M \\
& = \int_{\R} \bigg [-\varepsilon \frac{\Omega_0}{2} \left (D_{X}\widetilde{\zeta} + \frac{ic\widetilde{\mu}}{\Omega_0}\right )^{2} + \varepsilon^{2} \frac{\Omega_1}{2} \widetilde{\zeta} (D_{X}^{2}|D_{X}| \widetilde{\zeta}) + \varepsilon^{3} \frac{\Omega_2}{2} \widetilde{\zeta}(D_{X}^{4}\widetilde{\zeta}) \\
&\qquad \quad  + \frac{\varepsilon \varepsilon_{1}^{2}}{2} \omega_{1}''(k_{0})\overline{\q}(D_{X}^{2}\q) \bigg ]\, dX + H^{(3)} + \mathcal{O}\left(\varepsilon^4 + \frac{\varepsilon_1^4}{\varepsilon}\right).
\end{aligned}
\end{equation}
It is usual for Boussinesq systems to introduce $\nu := \partial_X \widetilde \zeta$, which is related to the horizontal velocity. Under \eqref{c0}-\eqref{k0} and \eqref{H3-in-tilde-variables}, the Hamiltonian \eqref{reduced-Hamiltonian-tilde-mu-q-2} becomes
\begin{equation}
\label{reduced-H-tilde-u-v1}
\begin{aligned}
\widehat{H} 
&=\int_{\R} \bigg [\frac{\varepsilon}{2} \left (\widetilde{\mu}-c_{0} \nu \right )^{2} + \varepsilon^{2} \frac{\Omega_1}{2} \nu |D_{X}|\nu - \varepsilon^{3} \frac{\Omega_2}{2} \nu(\partial_X^2 \nu) \\
&\qquad \quad  + \frac{\varepsilon \varepsilon_{1}^{2}}{2} \omega_{1}''(k_{0})\overline{\q}(D_{X}^{2}\q) - \varepsilon^{2} \kappa \widetilde{\mu} \nu^{2} + \varepsilon_{1}^{2} \left(\kappa_{1}\widetilde{\mu} + \kappa_{2}\nu \right)|\q|^{2}  \\
&\qquad \quad - \varepsilon^{3}\kappa_{3}\widetilde{\mu}(|D_{X}| \nu)(\nu) + \varepsilon \varepsilon_{1}^{2} \left( \kappa_{4}\widetilde{\mu} + \kappa_{5}\nu \right) \left(\q \overline{D_{X}\q} + \overline{\q} D_{X}\q \right) \\
&\qquad \quad + \varepsilon \varepsilon_{1}^{2} \left(\kappa_{6}(|D_{X}|\widetilde{\mu}) + \kappa_{7} (|D_{X}| \nu ) \right)|\q|^{2} - \varepsilon^{3}\kappa_{8}(|D_{X}|\widetilde{\mu}) \nu ^{2} \bigg ]\, dX + \mathcal{O}\left(\varepsilon^4 + \frac{\varepsilon_1^4}{\varepsilon}\right)
\end{aligned}
\end{equation}
and the equation of motion is 
\begin{equation}
\label{symplectic-form-tilde-mu-u-v_1}
\begin{aligned}
\partial_{t} \begin{pmatrix} \widetilde\mu \\ \nu \\ \q \\ \overline{\q} \end{pmatrix}
&= \begin{pmatrix} 0 & -\partial_X & 0 & 0 \\
-\partial_X & 0 & 0 & 0 \\
0 & 0 & 0 & -i\varepsilon \varepsilon_1^{-2} \\
0 & 0 & i\varepsilon \varepsilon_1^{-2} & 0
\end{pmatrix}
\begin{pmatrix} \delta_{\widetilde\mu} \widehat H \\ \delta_{\nu} \widehat H \\
\delta_{\q} \widehat H \\ \delta_{\overline{\q}} \widehat H \end{pmatrix}.
\end{aligned}
\end{equation}

\subsection{Benjamin-Ono -- Schr\"{o}dinger system} 
We adopt the characteristic variables to analyze the dynamics of the system in the preferred direction of propagation. Denoting the principally right-moving component of the solution by $r(X,t)$ and principally left-moving component by $s(X,t)$, we set 
\begin{equation}
\begin{pmatrix} r \\ s \end{pmatrix} 
= \begin{pmatrix} \frac{1}{\sqrt{2c_{0}}} & \sqrt{\frac{c_{0}}{2}} \\
\frac{1}{\sqrt{2c_{0}}} & -\sqrt{\frac{c_{0}}{2}} \end{pmatrix}
\begin{pmatrix} \widetilde{\mu} \\ \nu \end{pmatrix}.
\end{equation}
We focus on the propagation to the right by restricting our attention to $r(X,t)$ and assuming that $s(X,t)$ is of order $\mathcal{O}(\varepsilon^2)$.
Then, Hamiltonian \eqref{reduced-H-tilde-u-v1} is
\begin{equation}
\label{reduced-H-r-v1}
\begin{aligned} 
\widehat{H} &= \int_{\R} \bigg [\varepsilon^{2} \frac{\Omega_1}{4c_{0}}  r |D_{X}|r  - \varepsilon^{3} \frac{\Omega_2}{4c_{0}}r \partial_{X}^{2}r  - \frac{\varepsilon \varepsilon_{1}^{2}}{2}\omega_{1}''(k_{0})~ \overline{\q}~ \partial_{X}^{2}\q  \\
&\qquad \quad - \varepsilon^{2} \widetilde \kappa r^{3} + \varepsilon_{1}^{2} 
\widetilde \kappa_1 r|\q|^{2} - \varepsilon^{3} \widetilde \kappa_2 r^{2} |D_{X}|r   \\
&\qquad \quad + \varepsilon \varepsilon_{1}^{2} \widetilde \kappa_3 r \left(\q \overline{D_{X}\q} + \overline{\q} D_{X}\q \right) + \varepsilon \varepsilon_{1}^{2} \widetilde \kappa_4 (|D_{X}|r)|\q|^{2} \bigg ] \, dX + \mathcal{O}\left(\varepsilon^4 + \frac{\varepsilon_1^4}{\varepsilon}\right),
\end{aligned}
\end{equation}
where the new coefficients $\widetilde \kappa_j$ are defined as
\begin{equation}
\label{kappa-tilde-defn}
\begin{aligned}
&\widetilde \kappa = \frac{\kappa}{2\sqrt{2c_0}}, \quad \widetilde \kappa_1 = \kappa_1 \sqrt{\frac{c_0}{2}} + \kappa_2 \frac{1}{\sqrt{2c_0}}, \quad \widetilde \kappa_2 = \frac{1}{2\sqrt{2c_0}}(\kappa_3+\kappa_8), \\
& \widetilde \kappa_3 = \kappa_4 \sqrt{\frac{c_0}{2}} + \kappa_5 \frac{1}{\sqrt{2c_0}}, \quad 
\widetilde \kappa_4 = \kappa_6 \sqrt{\frac{c_0}{2}} + \kappa_7 \frac{1}{\sqrt{2c_0}}. 
\end{aligned}
\end{equation}
The corresponding evolution equations are
\begin{equation}
\label{symplectic-form-r-v_1}
\begin{aligned}
\partial_{t} \begin{pmatrix} r \\ s \\ \q \\ \overline{\q} \end{pmatrix}
&= \begin{pmatrix} -\partial_X & 0 & 0 & 0 \\
0 & \partial_X & 0 & 0 \\
0 & 0 & 0 & -i\varepsilon \varepsilon_1^{-2} \\
0 & 0 & i\varepsilon \varepsilon_1^{-2} & 0
\end{pmatrix}
\begin{pmatrix} \delta_r \widehat H \\ \delta_s \widehat H \\
\delta_{\q} \widehat H \\ \delta_{\overline{\q}}\widehat H \end{pmatrix}.
\end{aligned}
\end{equation}
Then, the equation of motion for $r(X, t)$ satisfies the Benjamin-Ono (BO) equation expressed in a moving reference frame, $\partial_t r = - \partial_X \delta_r \widehat H$. Over the time scale $\tau = \varepsilon t$, the equation becomes
\begin{equation}
\label{benjamin-ono}
\begin{aligned}
\partial_{\tau}r
&= - \varepsilon \frac{\Omega_1}{2c_{0}} \partial_{X}(|D_{X}|r) + 6 \varepsilon \widetilde \kappa r(\partial_{X}r) + \varepsilon^2 \frac{\Omega_2}{2 c_{0}} \partial_{X}^{3}r \\
&\quad - \varepsilon^{1+2\delta}\widetilde \kappa_1 \partial_{X}(|\q|^{2}) + \varepsilon^2 \widetilde \kappa_2 \left( \partial_{X} (r |D_{X}|r) + |D_{X}|(r\partial_{X}r) \right) \\
&\quad - \varepsilon^{2+2\delta}\widetilde \kappa_3 \partial_{X} \left(\q \overline{D_{X}\q} + \overline{\q} D_{X}\q \right) - \varepsilon^{2+2\delta}\widetilde \kappa_4 \partial_{X} |D_{X}| \left( |\q|^{2} \right),
\end{aligned}
\end{equation}
where we used that $\varepsilon_1 = \varepsilon^{1+\delta}$ as in \eqref{mono-waves-scaling}.

Similarly, we write the equation of motion for the variable $\q(X,t)$ over the time scale $\tau$, 
\begin{equation}
\label{schrodinger}
\begin{aligned}
i\partial_{\tau} \q 
&= - \varepsilon \frac{\omega_{1}''(k_{0})}{2} \partial_{X}^{2}\q + 
\widetilde \kappa_1 r \q  - i \varepsilon \widetilde \kappa_3 (\partial_{X}(r\q) + r\partial_{X}\q) + \varepsilon \widetilde \kappa_4 \q |D_{X}|r,
\end{aligned}
\end{equation}
which is of a Schr\"{o}dinger type. 

\begin{remark}
\label{remark-scaling}
It may be more natural to choose the time scale $\tau_{1} = \varepsilon^{2}t$ for the Benjamin-Ono equation \eqref{benjamin-ono}, so that the contribution of leading terms on the right-hand side of the equation is of order $\mathcal{O}(1)$. However, we intentionally choose the same time scale $\tau = \varepsilon t$ in  \eqref{benjamin-ono}-\eqref{schrodinger}
since both equations contain coupling terms and the current form of the system is more suitable for further analysis of its local well-posedness in the next section.
\end{remark}

\begin{remark}
\label{remark-delta-values}
We comment on the restriction $\delta \in \left(0, \frac{1}{2}\right)$ used in the definition of $\varepsilon_1$ in \eqref{mono-waves-scaling}. The main reason for this choice is to ensure that terms of order $\varepsilon \varepsilon_1^2$ will not be absorbed into the remainder part $\mathcal{O} \left( \varepsilon^4 + \frac{\varepsilon_1^4}{\varepsilon} \right)$ in the Hamiltonian \eqref{reduced-H-r-v1}. Indeed, if $\delta = 0$ (which is $\varepsilon_1 = \varepsilon$), then $\varepsilon \varepsilon_1^2 =  \frac{\varepsilon_1^4}{\varepsilon}$. On the other hand, if $\delta \geq 1/2$ then 
$\varepsilon \varepsilon_1^2 = \mathcal{O} (\varepsilon^4)$.
\end{remark}

\begin{remark}
One can further consider the system \eqref{benjamin-ono}-\eqref{schrodinger} under the regime $\rho \approx \rho_1$. The interest is motivated by the fact that $\rho_1 / \rho \approx 0.99$ is a typical value of the density ratio in realistic conditions. For example, according to the measurements in \cite{OB80} for the Andaman Sea, $\rho_1 / \rho = 0.997$ and in \cite{BO02} off the Oregon coast,  $\rho_1 / \rho = 0.998$. We do not pursue such investigation here; however, we note that the coefficients $\widetilde \kappa$ and $\widetilde \kappa_j$ defined in \eqref{kappa-tilde-defn} get significantly simpler. We provide estimates on the $\widetilde \kappa$ and $\widetilde \kappa_j$ in this regime in Lemma \ref{lemma-kappa-asymptotics}.    

\end{remark}

\section{Local well-posedness of the reduced system}
\label{section-lwp}

\subsection{Setting of Problem}

We first review some local and global well-posedness results on the BO equation and related systems. 
We start with the Cauchy problem for the classical BO equation
\begin{align}
v_{t} + \mathcal{H}(v_{xx}) = vv_{x} \, ; \, v(x,0) = v_{0}(x),
\end{align}
where $\H$ is the Hilbert transform defined in \eqref{hilbert-transform-defn}. The equation is integrable by inverse scattering \cite{AF83}  and has an infinite number of conserved quantities, the first three being,  
 the $L^{2}$-norm
\begin{align}
\int_{\R} v(x,t)^{2} \, dx = \int_{\R} v(x,0)^{2} \, dx,
\end{align}
 the Hamiltonian
\begin{align}
\int_{\R} \left (v \mathcal{H}(v_{x}) - \frac{1}{3}v^{3} \right ) \, dx,
\end{align}
and  the next one in the hierarchy 
\begin{align}
\int_{\R} \left (v_{x}^{2} - \frac{3}{4}v^{2}\mathcal{H}(v_{x}) - \frac{1}{8}v^{4} \right ) \, dx.
\end{align}
A broad survey of PDE and inverse scattering  
methods for  the Cauchy problem of the Benjamin-Ono equation and Intermediate Long Wave equations can be found  in Chapter 3 of Klein-Saut \cite{KS21}.
\newline \\

The question of well-posedness of the Cauchy problem has been studied extensively for progressively rougher classes of initial data in $H^{s}(\R)$. 
Local wellposedness of solutions in $H^s$, 
$s>\frac{3}{2}$ was obtained by  Iorio \cite{Ior86}, extended to  global  solutions in $s \geq \frac{3}{2}$ by Ponce \cite{Pon91}. Using dispersive estimates, Koch and Tzvetkov \cite{KT03} extended the local well-posedness to 
 $s>\frac{5}{4}$ and Kenig \cite{KK03} further improved it to  $s>\frac{9}{8}$.
By performing a gauge transformation to eliminate high-order derivatives in the non-linear term and applying  Strichartz estimates, Tao \cite{T04} further shows global  wellposedness in $H^s$, $s\ge 1$, where the global result is due to  the invariant at the level of $H^{1}$. Later, Burq and Planchon established well-posedness locally for $s>\frac{1}{4}$ and globally for $s>\frac{1}{2}$ \cite{BP05}. Lastly, global well-posedness was established for $s \geq 0$ by  Ionescu and Kenig \cite{IK07} and Molinet and Pilod \cite{MP12a}.

\bigskip
Analysis of the  higher order BO equation  brings serious difficulties.  Linares, Pilod and Ponce \cite{LPP11} proved local well-posedness for $H^{2}$ initial data of the Cauchy problem for the following higher-order BO equation
\begin{align}
\partial_{t}v - b \mathcal{H}(v_{xx}) + a v_{xxx} = cvv_{x} - d\left (v\mathcal{H}(v_{x}) + \mathcal{H}(vv_{x}) \right )_{x} \, ; v(0,x)=v_{0}(x).
\end{align}
For lack of conserved quantities at the level of $H^{2}$, a global result was not established. However, a little later, Molinet and Pilod   \cite{MP12b} proved  global well-posedness  for $H^{1}$ initial data. The main effort was to prove existence locally in $H^{1}$ and then they  used the conserved quantity at the level of $H^{1}$ to obtain global well-posedness. 

\bigskip
Turning to coupled systems, the BO-Schr\"odinger  system
\begin{align}
\begin{cases}
&i\partial_{t}u + \partial_{x}^{2}u = \alpha uv \\
&\partial_{t}v +\nu \partial_{x}(|D_{x}|v) = \beta \partial_{x}(|u|^{2})
\end{cases}
\end{align}
was derived by Funakoshi and Oikawa \cite {FO83} in a regime where the internal wave has a smaller amplitude than the surface one. In particular, the BO equation does not have nonlinear terms.

This system has the following three conservation laws:
\begin{subequations}
\begin{align}
 &   M= \int_\R |u(x,t)|^2 dx \\
 &   L= -\frac{\alpha}{2\beta} \int_\R |v(x,t)|^2 dx -\Im \int_\R u\overline{u}_{x} dx \\
 &   H= \int_\R  \Big(\big(|u_x|^2 -\frac{\alpha \nu}{2\beta} ||D_x|^{1/2}v|^{2} + \alpha v |u|^{2} \Big) dx.
\end{align}
\end{subequations}
In the context of several classes of systems (KdV-NLS, BO-NLS), Bekiranov, Ogawa and Ponce \cite{BOP98} proved local existence in $H^s\times H^{s-1/2} $, $\nu\ne 1$, $s\ge0$. Because of conservation laws, the solutions persist globally in time when $s\geq 1$, $ \alpha \nu/\beta<0$.
Pecher \cite{P06} extended  the global  result to   $|\nu|=1, \alpha/\beta<0$, and  $s>1/3$. 
Angulo, Matheus and Pilod \cite{AMP09} further improved global well-posedness to $s=0$,
$\nu\ne0, |\nu| \ne 1$.
Finally, we mention  a recent article of Linares, Mendez and  Pilod  \cite{LMP23} who included nonlinear terms in the Schr\"odinger and BO equations  and proved local existence of solutions
in $H^{s+1/2}\times H^{s}$, $s> 5/4$.

Our  system \eqref{benjamin-ono}-\eqref{schrodinger} derived in the previous section combines two difficult issues, high order terms in both equations  and coupling. For this reason,
we will restrict ourselves to a simplified version of it, where we neglect the last two terms of each equation,  namely we will consider the system 
\begin{equation}
\label{reduced-system}
\left\{ \begin{aligned}
& \partial_t r + a \partial_x^3 r - b \H \partial_x^2 r = c r \partial_x r - d \partial_x \left( r \H \partial_x r + \H (r \partial_x r) \right) + \beta \partial_x \left( |q|^2 \right), \\
& i \partial_t q -\alpha \partial_x^2 q = - \beta qr,\\
&r(x, 0) = r_0(x), \quad q(x,0) = q_0(x),
\end{aligned} \right.
\end{equation}
where $x, t \in \mathbb{R}$, $v$ is a real-valued function, $a$ is a nonzero real number, $b, c, d, \f,\alpha, \beta$ are positive real numbers. 
Solutions  of \eqref{reduced-system} conserve the Hamiltonian (energy)
\begin{equation*}
\label{energy}
\begin{aligned} 
E_1 &: = \int_{\R} \left( - \frac{b}{2} r |D_{x}|r  - \frac{a}{2} |\partial_{x} r|^2 - \alpha |\partial_{x} q|^2 -\frac{c}{6} r^{3} - \beta
r |q|^{2} + \frac{d}{2} r^{2} |D_{x}|r \right) dx
\end{aligned}
\end{equation*}
and the $L^2$-norm of the solution $q$,
\begin{equation*}
E_2 := \int_\mathbb{R} |q|^2 dx . 
\end{equation*}
The $L^2$-norm of $r$, however, is not preserved. Instead, we have the conservation of the quantity
\begin{equation*}
E_3 := \frac{1}{2} \int_\mathbb{R} r^2 dx +  \Im \int_\mathbb{R} q  \overline{\partial_{x} q} dx. 
\end{equation*}

Our main  well-posedness result for the system \eqref{reduced-system} is the following: 
\begin{theorem}
\label{well-posedness-theorem}
Let 
\begin{equation}
\label{space-X}
X := \left( H^2(\mathbb{R}) \cap L^2(\mathbb{R}; x^2 dx) \right) \times \left( H^3(\mathbb{R}) \cap L^2(\mathbb{R}; x^2 dx) \right).
\end{equation}
For every $(r_0, q_0) \in X$, there exists a positive time $T = T(\|(r_0, q_0)\|_{X})$
for which the initial value problem \eqref{reduced-system} admits a unique solution $(r,q) \in X$ 
satisfying
\begin{equation}
\label{solution-properties}
\begin{aligned}
(r,q) & \in C \left([0, T]: X \right), \\
\partial_x^l r & \in L^6\left( [0, T]: L_x^\infty (\mathbb{R}) \right) \quad \text{for  } l=0,1,2,\\
\partial_x^3 r & \in L^2\left( [0, T]: L_{loc}^2 (\mathbb{R}) \right), \\
\partial_x (xr) & \in L^2\left( [0, T]: L_{loc}^2 (\mathbb{R}) \right). 
\end{aligned}
\end{equation}
\end{theorem}
We point out that Theorem \ref{well-posedness-theorem} implies the boundedness of the weighted $L^2_x$-norms of $r$ and $q$, which are necessary to control the gauge transformation \eqref{gauge-function}. The additional properties of $r$ in \eqref{solution-properties} appear as a result of boundedness of the semi-norms \eqref{lambda-seminorms} used in the proof.

\subsection{Strategy of the proof} 

The proof is based on a standard contraction mapping principle for appropriately chosen norms, which we will discuss in more details below. 
The main difficulty in the analysis of \eqref{reduced-system} is the presence of second-order derivatives of $r$  in the nonlinear part of the Benjamin-Ono  equation. Such terms cannot be treated using smoothing estimates as in Lemma
\ref{lemma-semigroup-estimates} since the second inequality in \eqref{estimates-BO-operators} can handle only up to one derivative. 
To overcome this difficulty,  we introduce pseudo-differential operators $P_+$ and $P_-$, which correspond to projections into positive and negative frequencies, respectively: 
$$
\widehat{P_\pm f} (k) = {1}_{\mathbb{R}_\pm} \hat f(k).
$$
Then, the relations
\begin{equation}
\label{projections-relations}
\H = -i (P_+ - P_-) \quad \textrm{and} \quad 1 = P_+ + P_-. 
\end{equation}
allow to apply $P_\pm$ to the BO equation, 
see \eqref{equation-for-dxPv-plus} and \eqref{equation-for-dxPv-minus}. The high-derivative nonlinear terms of these equations consist of terms $r \partial_x^3 P_\pm r$ and terms involving the commutator $[P_\pm, r]$. The most problematic out of these two is $r \partial_x^3 P_\pm r$, since the other term can be treated using a commutator estimate from Lemma \ref{lemma-commutator} proved in Dawson, McGahagan and Ponce \cite{DMP07}. To remove the problematic term, we use the idea of a gauge transformation, initially introduced by Hayashi and Ozawa in \cite{HO94} in the context of the NLS equation and  applied by Tao in \cite{T04} and Linares, Pilod and Ponce in \cite{LPP11} for BO equation. Our gauge transformation function $\Psi$, defined in \eqref{gauge-function}, is constructed as in \cite{LPP11}. 
As a result of these algebraic steps, we rewrite the BO equation as a system of three nonlinear PDEs for $r$ and $w_\pm := \Psi_\pm \partial_x P_\pm r$, see \eqref{final-system}. 

The proof of Theorem \ref{well-posedness-theorem} now follows from solving the system \eqref{final-system} coupled with the Schr\"odinger equation in \eqref{reduced-system}  and  using a fixed point argument. Namely, we intend to solve the system
\begin{equation}
\label{reduced-system-extended}
\left\{ \begin{aligned}
\left( \partial_t + a \partial_x^3 - b \H  \partial_x^2 \right) r & =\mathcal{N}(r, w_\pm) + \f \partial_x \left(|q|^2 \right),\\
(\partial_t + a \partial_x^3 + ib \partial_x^2) w_+ & = \mathcal{M}_+ (r, w_\pm) + 2id \Psi_+ [P_+, r] \partial_x^2 (\Psi_-w_+) + \f \Psi_+ \partial_x^2 P_+ \left( |q|^2 \right),\\
(\partial_t + a \partial_x^3 - ib \partial_x^2) w_- & = \mathcal{M}_- (r, w_\pm) - 2id \Psi_- [P_-, r] \partial_x^2 (\Psi_+ w_-) + \f \Psi_- \partial_x^2 P_- \left( |q|^2 \right), \\
i \partial_t q -\alpha \partial_x^2 q & = - \beta qr,\\
r(x, 0) = r_0(x), \quad & q(x,0) = q_0(x), \quad w_\pm (x, 0) = w_{\pm, 0} : = \Psi_\pm \partial_x P_\pm r_0. 
\end{aligned} \right.
\end{equation}

\begin{proposition}
\label{well-posedness-proposition}
Let 
\begin{equation*}
Y := \left( H^1(\mathbb{R}) \cap L^2(\mathbb{R}; x^2 dx) \right) \times H^1(\mathbb{R})^2 \times \left( H^3(\mathbb{R}) \cap L^2(\mathbb{R}; x^2 dx) \right).
\end{equation*}
For every $(r_0, w_{\pm, 0}, q_0) \in Y$, there exists $T = T(\|(r_0, w_{\pm, 0},  q_0)\|_{Y})$ 
for which the initial value problem \eqref{reduced-system-extended} admits a unique solution $(r, w_\pm, q) \in Y$ satisfying
\begin{equation*}
\label{extended-solution-properties}
\begin{aligned}
(r,w_\pm, q) & \in C \left([0, T]: Y \right), \\
\partial_x^l r, \partial_x^l w_\pm & \in L^6\left( [0, T]: L_x^\infty (\mathbb{R}) \right) \quad \text{for  } l=0,1,\\
\partial_x^2 r, \partial_x^2 w_\pm  & \in L^2\left( [0, T]: L_{loc}^2 (\mathbb{R}) \right), \\
\partial_x (xr) & \in L^2\left( [0, T]: L_{loc}^2 (\mathbb{R}) \right). 
\end{aligned}
\end{equation*}
\end{proposition}
From the definition \eqref{w-plus-defn}, it is clear that Proposition \ref{well-posedness-proposition} implies Theorem \ref{well-posedness-theorem}. Indeed, one can relate the functions $r$ and $w_\pm$ via
\begin{equation*}
\partial_x r = \Psi_- w_+ + \Psi_+ w_-. 
\end{equation*}
In Section \ref{section-fixed-point}, we provide the proof of Proposition \ref{well-posedness-proposition}, which follows from the fixed point argument applied to the integral form of the equations \eqref{reduced-system-extended}. The integral forms are given in \eqref{integral-system}--\eqref{integral-system-weighted}. The semi-norms we use are similar to those in \cite{LPP11}; they are given in \eqref{lambda-seminorms}--\eqref{mu-seminorms}
with an additional semi-norm needed to get the $H_x^3$-norm estimates for the solution $q$ of the Schr\"odinger equation. 
We point out that 
the norms $\lambda_3^T$ and $\lambda_4^T$ involving the cut-off function \eqref{cut-off-definition} are needed to control the terms involving $\partial_x^2 v$ and $\partial_x^2 w_\pm$. The example of such treatment is given in \eqref{second-term-lemma-nonl}, where we face $\partial_x^2 v$ appearing from $\left\| \mathcal{N} \right\|_{H_x^1}$ in \eqref{main-bound}.   
In view of \eqref{der-order-reduction-dx2v}, it may be also possible to control $\partial_x^2 v$ by means of $\lambda_1^T(w_\pm)$ or $\lambda_2^T(w_\pm)$ as in \eqref{commutator-lemma-term-estimate}; however, similar ideas seem to fail for $\partial_x^2 w_\pm$. 

\subsection{Preliminary lemmas}

Throughout the rest of the paper, we denote $\langle t \rangle := \sqrt{1+t^2}$. We write $a \lesssim b$ if there is a uniform constant $C$ such that $a \leq Cb$. If $a \lesssim b$ and $b \lesssim a$, we write $a \approx b$.  

We introduce a cut-off function $\chi(x) \in C_{c}^{\infty}(\R)$ such that $0 \leq \chi \leq 1$, $\chi \equiv 1$ on $[0,1]$, $\mathrm{supp} \, \chi \subset (-1,3)$ and set
\begin{equation}
\label{cut-off-definition}
\chi_{j/N}(x) := \chi \left (\frac{x-j}{N} \right )
\end{equation}
for $N \in \mathbb{Z}_+$. Note that, by construction, 
\begin{equation}
\label{sum-of-cut-offs}
\sum_{j=-\infty}^{+\infty} \chi_{j/N}(x) \approx N \quad \text{ for all } x\in \mathbb{R}.     
\end{equation}

The analysis of the equations in \eqref{reduced-system} involves linear operators $(a\partial_x^3 - b \H \partial_x^2)$, $(a\partial_x^3 \pm i b \partial_x^2)$ and $i \alpha \partial_x^2$, so we define the corresponding unitary groups in $H^s(\mathbb{R})$ as follows:
\begin{equation}
\label{semigroups}
V (t) := e^{-t(a\partial_x^3 - b \H \partial_x^2)}, \quad
W_\pm (t) := e^{-t(a\partial_x^3 \pm i b \partial_x^2)} \quad \text{and} \quad U(t) := e^{-i t \alpha \partial_x^2}.
\end{equation}
\begin{lemma} { \rm{(Strichartz estimates)}  }
\label{lemma-semigroup-estimates} 
Let $S$ be one of the following semi-groups: $V$ or $W_\pm$. Then, for every $T>0$ and every function $h$ of sufficient regularity in $x$, we have 
\begin{equation}
\label{estimates-BO-operators}
\begin{aligned}
\left\| S(t) h \right\|_{L_t^6 L_x^\infty} & \lesssim \|h\|_{L_x^2}, \\
\left( \sup\limits_{j \in \mathbb{Z}} \int_0^T \int_{\mathbb{R}} \left| \chi_{j/N} (x) \partial_x S(t) h(x) \right|^2 dx dt \right)^{1/2} & \lesssim \langle T \rangle^{1/2} \|h\|_{L_x^2},\\
\left( \sum\limits_{j \in \mathbb{Z}} \sup\limits_{0\leq t \leq T} \sup\limits_{x \in \mathbb{R}} \left| \chi_{j/N} (x) S(t) h(x) \right|^2 \right)^{1/2} & \lesssim  \langle T \rangle^2 \|h\|_{H_x^1},
\end{aligned}
\end{equation}
where $\chi_{j/N}$ is given by \eqref{cut-off-definition} and $N = N(T) \approx 1+T$. 

Moreover, for the semi-group $U(t)$ we have 
\begin{equation}
\label{estimates-Schr-operator}
\begin{aligned}
\left( \sum\limits_{j \in \mathbb{Z}} \sup\limits_{0\leq t \leq T} \sup\limits_{x \in \mathbb{R}} \left| \chi_{j/N} (x) U(t) h(x) \right|^2 \right)^{1/2} & \lesssim  \langle T \rangle^2 \|h\|_{H_x^1}.
\end{aligned}
\end{equation}
\end{lemma}

\begin{proof}
For the proof of \eqref{estimates-BO-operators}, we refer to Lemma 2.1 in \cite{LPP11} and references therein. For the proof of \eqref{estimates-Schr-operator}, we refer to Theorem 3.1 in \cite{KPV93} and references therein.
\end{proof}


The commutators involving the two projection operators
\begin{equation} \label{commutator}
[P_\pm, h] := P_\pm h - h P_\pm
\end{equation}
satisfy the following estimate proven in \cite{DMP07}.
\begin{lemma} {\rm{(Commutator lemma)}  }
\label{lemma-commutator} 
Let $h$ be a function of sufficient regularity, 
then for any $p>1$ and nonnegative integers $l, m$ with $l + m \geq 1$, there exists a constant $c = c(p,l,m)$ such that 
\begin{equation}
\label{commutator-estimate}
\left\| \partial_x^l [P_\pm, h] \partial_x^m f \right\|_{L_x^p} \leq c \| \partial_x^{(l+m)} h \|_{L_x^\infty} \|f\|_{L_x^p}.
\end{equation}
\end{lemma}
\noindent We will also need the following result, which shows the commutativity property of a space variable $x$ with a linear operator associated with $V$ defined in \eqref{semigroups}. 
\begin{lemma}  {\rm(Commutativity relation)}
\label{lemma-commutativity}
For all $h \in \mathcal{S} (\mathbb{R})$, we have
\begin{equation}
\begin{aligned}
\label{commutativity-relation}
(a \partial_x^3 - b \H \partial_x^2) (xh) = (3a\partial_x^2 - 2b\H \partial_x) h + x (a \partial_x^3 - b \H \partial_x^2) h.
\end{aligned}
\end{equation}
\end{lemma}
\begin{proof}
The computations for $\partial_x^3$ part are straightforward. For $\H \partial_x^2$, we derive the estimates using the Fourier space. 
Indeed, using the integration by parts, we can write 
$$
\begin{aligned}
\H \partial_x^2 (xh) &= \int i k^2 \sgn (k) ~\widehat{xh}(k) e^{ikx} dk\\
& = \int \partial_k \left( |k|k e^{ikx} \right) \widehat{h} (k) dk\\
& = \int \left( 2k \sgn(k) + ix |k| k \right) e^{ikx} \widehat{h} ~dk = 2 \partial_x \H h + x \partial_x^2 \H h. 
\end{aligned}
$$
\end{proof}

\subsection{Reformulation of the Benjamin-Ono equation} 

The goal of this section is to rewrite the Benjamin-Ono equation as a system of three partial differential equations for dependent variables $r$ and $w^{\pm}$, where $w^{\pm}$ are defined in terms of $\partial_{x}r$. Ultimately, the non-linear terms in the system for $(v,w^{\pm})$ only involve first-order derivatives as shown in \eqref{equation-for-wplus}, \eqref{der-order-reduction-dx2v} and \eqref{equation-for-wminus}.
\newline \\
\noindent Applying the derivative $\partial_x$ to the first equation in \eqref{reduced-system}, we get
\begin{equation}
\label{equation-dxv}
\begin{aligned}
\left( \partial_t + a \partial_x^3 - b \H \partial_x^2 \right) \partial_x r = &~c \left( (\partial_x r)^2 + r \partial_x^2 r \right) - d \left( (\partial_x^2 r) (\H \partial_x r) + 2 (\partial_x r) (\H \partial_x^2 r) \right)\\
& - 3 d \H \left( (\partial_x^2 r) (\partial_x r) \right) - d \left( r (\H \partial_x^3 r) + \H (r \partial_x^3 r) \right) + \f \partial_x^2 \left( |q|^2 \right)
\end{aligned}
\end{equation}
The Hilbert transform operator in the above equation can be expressed using the relation \eqref{projections-relations}, and projecting the whole equation onto $P_+$, we get 
\begin{equation}
\label{equation-for-dxPv-plus}
\left( \partial_t + a \partial_x^3 + i b ~ \partial_x^2 \right) \partial_x P_+ r = P_+ Q_+ (r) + 2id r~ \partial_x^3 P_+ r + 2id [P_+, r] \partial_x^3 P_+ r + \f \partial_x^2 P_+ \left( |q|^2 \right),
\end{equation}
where $[P_+, r]$ is the commutator defined in \eqref{commutator} and 
\begin{equation*}
Q_+(r) := c \left(  (\partial_x r)^2 + v \partial_x^2 v \right) - d \left( (\partial_x^2 r) (\H \partial_x r) + 2 (\partial_x r) (\H \partial_x^2 r) \right) + 3id (\partial_x r) (\partial_x^2 r).
\end{equation*}
Projecting \eqref{equation-dxv} onto $P_-$, we get the equation for $\partial_x P_- v$, which can be obtained from \eqref{equation-for-dxPv-plus} by taking its complex conjugate,
\begin{equation}
\label{equation-for-dxPv-minus}
\left( \partial_t + a \partial_x^3 - i b ~ \partial_x^2 \right) \partial_x P_- r = P_- \overline{Q_+} (r) - 2id r~ \partial_x^3 P_- r - 2id [P_-, r] \partial_x^3 P_- r + \f \partial_x^2 P_- \left( |q|^2 \right).
\end{equation}
To eliminate the higher-order derivative term $v\partial_x^3 P_+ r$ in \eqref{equation-for-dxPv-plus}, we multiply the equation by a function $\Psi = \Psi (x, t)$, which will be chosen later. Then using the identities 
\begin{equation*}
\begin{aligned}
\Psi \partial_t \partial_x P_+ r & = \partial_t (\Psi \partial_x P_+ r) - (\partial_t \Psi) (\partial_x P_+ r),\\
a \Psi \partial_x^3 \partial_x P_+ r & = a \partial_x^3 (\Psi \partial_x P_+ r) - 3a (\partial_x \Psi) (\partial_x^3 P_+ r) - 3a (\partial_x^2 \Psi) (\partial_x^2 P_+ r) - a (\partial_x^3 \Psi) (\partial_x P_+ r),\\
ib \Psi \partial_x^2 \partial_x P_+ r & = i b \partial_x^2 (\Psi \partial_x P_+ r) - 2i b (\partial_x \Psi) (\partial_x^2 P_+ r) - ib (\partial_x^2 \Psi) (\partial_x P_+ r),
\end{aligned}
\end{equation*}
we get 
\begin{equation}
\label{equation-for-PsidxPv}
\begin{aligned}
(\partial_t + a \partial_x^3 + ib \partial_x^2) \Psi \partial_x P_+ r =&~ (\partial_t \Psi) (\partial_x P_+ r) + 3a (\partial_x \Psi) (\partial_x^3 P_+ r) + 3a (\partial_x^2 \Psi) (\partial_x^2 P_+ r) \\
& + a (\partial_x^3 \Psi) (\partial_x P_+ r) + 2i b (\partial_x \Psi) (\partial_x^2 P_+ r) + ib (\partial_x^2 \Psi) (\partial_x P_+ r)\\
& + \Psi P_+ Q_+ (r) + 2id \Psi v~ \partial_x^3 P_+ v + 2id \Psi [P_+, r] \partial_x^3 P_+ r \\
& + \f \Psi \partial_x^2 P_+ \left( |q|^2 \right). 
\end{aligned}
\end{equation}
The higher-order derivative terms involving $\partial_x^3 P_+ r$ will vanish if we choose a function $\Psi$ that satisfies  
\begin{equation}
\label{gauge-function}
3a \partial_x \Psi + 2id \Psi r = 0 \implies 
\Psi (x,t) = e^{-\frac{2id}{3a} \int_{-\infty}^x r(y,t) dy}.
\end{equation}
The term $\partial_t \Psi$ in \eqref{equation-for-PsidxPv} can be found from \eqref{gauge-function} as 
\begin{equation*}
\begin{aligned}
\partial_t \Psi & = - \frac{2id}{3a} \Psi \int_{-\infty}^x r_t dy\\
& = - \frac{2id}{3a} \Psi \left( -a \partial_x^2 r + b\H \partial_x r + \frac{c}{2} r^2 - dr (\H \partial_x r) - d \H (r\partial_x r) \right),
\end{aligned}
\end{equation*}
where we used \eqref{reduced-system} to define $r_t$. 

Setting $\Psi_{+} = \Psi$ and $\Psi_{-} = \overline{\Psi}$, we define the function on left-hand side of \eqref{equation-for-PsidxPv} as 
\begin{equation}
\label{w-plus-defn}
w_+ := \Psi_+ \partial_x P_+ r.
\end{equation}
We also define $w_- :=  \Psi_- \partial_x P_- r  =\overline{w}_+ $  
whose evolution equation is given by the complex conjugate of \eqref{equation-for-PsidxPv}. 
Under the condition \eqref{gauge-function}, the equation \eqref{equation-for-PsidxPv} becomes 
\begin{subequations}
\begin{align}
(\partial_t + a \partial_x^3 + ib \partial_x^2) w_+ & = - \frac{2id}{3a} w_+ \left( -a \partial_x^2 v + b\H \partial_x r + \frac{c}{2} r^2 - dr (\H \partial_x r) - d \H (r\partial_x r) \right) \nonumber \\
& ~+ 3a (\partial_x^2 \Psi_+) (\partial_x^2 P_+ r) + a (\partial_x^3 \Psi_+) (\partial_x P_+ r) + 2i b (\partial_x \Psi_+) (\partial_x^2 P_+ r) \nonumber \\
& ~ + ib (\partial_x^2 \Psi_+) (\partial_x P_+ r) + \Psi_+ P_+ Q_+ (r) + 2id \Psi_+ [P_+, r] \partial_x^3 P_+ r \nonumber \\
& ~ + \f \Psi_+ \partial_x^2 P_+ \left( |q|^2 \right) \\
& =: \mathcal{M}_+ (r, w_\pm) + 2id \Psi_+ [P_+, r] \partial_x^2 (\Psi_-w_+) + \f \Psi_+ \partial_x^2 P_+ \left( |q|^2 \right) \label{equation-for-wplus},
\end{align}
\end{subequations}
 where we write $\partial_x^3 P_+ r = \partial_x^2 (\Psi_-w_+)$ and $\mathcal{M}_{+}$ denotes the sum of all terms in \eqref{equation-for-wplus} except the last two. Note that, based on the definitions of $w_\pm$ and $\Psi$ in \eqref{gauge-function}, the nonlinear function $\mathcal{M}_+$ can be expressed as a polynomial involving terms in $r$ and $w_\pm$ up to their first derivatives only.
Indeed, the second order derivative terms appearing in \eqref{equation-for-wplus} are treated similar to the following
\begin{equation}
\label{der-order-reduction-dx2v}
\begin{aligned}
& \partial_x^2 r = \partial_x \left( (P_+ + P_-) \partial_x r \right) = \partial_x \left(\Psi_- \Psi_+ \partial_x P_+ r + \Psi_+ \Psi_- \partial_x P_- r \right) = \partial_x \left(\Psi_- w_+ + \Psi_+ w_- \right),
\end{aligned}
\end{equation}
while for the terms involving $\partial_x^3 \Psi$ we use \eqref{gauge-function} as well. 
Taking the complex conjugate of \eqref{equation-for-wplus}, we get an equation for $w_-$ given by
\begin{equation}
\label{equation-for-wminus}
(\partial_t + a \partial_x^3 - ib \partial_x^2) w_- = \mathcal{M}_- (r, w_\pm) - 2id \Psi_- [P_-, r] \partial_x^2 (\Psi_+ w_-) + \f \Psi_- \partial_x^2 P_- \left( |q|^2 \right),
\end{equation}
where $\mathcal{M}_- = \overline{\mathcal{M}_+}$.
The equation \eqref{equation-for-wplus} is now ready for further analysis. 

We also rewrite the original Benjamin-Ono equation in \eqref{reduced-system} using the newly defined function $w_\pm$. We note that its higher-order derivative terms in the nonlinear part can also be simplified as in \eqref{der-order-reduction-dx2v} using identities  \eqref{projections-relations} and \eqref{gauge-function}. For example, 
\begin{equation*}
\begin{aligned}
\partial_x (r \H \partial_x r ) = -i \partial_x \left( r \Psi_+ \Psi_- (P_+ - P_-) \partial_x r \right) = -i \partial_x \left( r (\Psi_- w_+ - \Psi_+ w_-) \right).
\end{aligned}
\end{equation*}
Repeating similar manipulations for the rest of higher-order nonlinear terms, the equation for $v$ in \eqref{reduced-system} becomes 
\begin{equation}
\label{equation-for-v}
\begin{aligned}
\left( \partial_t + a \partial_x^3 - b \H  \partial_x^2 \right) r & = c v \partial_x r + id \partial_x(r\Psi_-w_+ - v\Psi_+ w_-) - d \H \partial_x(r\Psi_-w_+ + r\Psi_+ w_-)\\
& \quad + \f \partial_x \left(|q|^2 \right)\\
& := \mathcal{N}(r, w_\pm) + \f \partial_x \left(|q|^2 \right),
\end{aligned}
\end{equation}
where $\mathcal{N}$ is a polynomial involving terms in $v$ and $w_\pm$ up to their first derivatives only.
As a result of the above steps, the Benjamin-Ono equation in \eqref{reduced-system} is equivalent to the following system of equation 
\begin{equation}
\label{final-system}
\left\{
\begin{aligned}
\left( \partial_t + a \partial_x^3 - b \H  \partial_x^2 \right) r & =\mathcal{N}(r, w_\pm) + \f \partial_x \left(|q|^2 \right),\\
(\partial_t + a \partial_x^3 + ib \partial_x^2) w_+ & = \mathcal{M}_+ (r, w_\pm) + 2id \Psi_+ [P_+, r] \partial_x^2 (\Psi_-w_+) + \f \Psi_+ \partial_x^2 P_+ \left( |q|^2 \right),\\
(\partial_t + a \partial_x^3 - ib \partial_x^2) w_- & = \mathcal{M}_- (r, w_\pm) - 2id \Psi_- [P_-, r] \partial_x^2 (\Psi_+ w_-) + \f \Psi_- \partial_x^2 P_- \left( |q|^2 \right)
\end{aligned}\right.
\end{equation}

Finally, we write the evolution equation for $xr$ and $xu$, which will be used later to get estimates of $v$ and $u$ in the weighted space $L^2(\mathbb{R}; x^2 dx)$.
We multiply \eqref{equation-for-v} by $x$, and using \eqref{commutativity-relation}, we can write
\begin{equation*}
\begin{aligned}
\left( \partial_t + a\partial_x^3 - b \H \partial_x^2 \right) (xr) = (3a \partial_x^2 - 2 b\H \partial_x) r + x\mathcal{N}(r, w_\pm) + \beta x \partial_x\left( |q|^2 \right). 
\end{aligned}
\end{equation*}
Similarly, we obtain 
\begin{equation*}
\begin{aligned}
\left( \partial_t + i \alpha \partial_x^2 \right) (xq) = 2i \alpha \partial_x q + i\beta xqr. 
\end{aligned}
\end{equation*}


\section{Proof of Proposition \ref{well-posedness-proposition} 
}
\label{section-fixed-point}
Equations for $r, w_\pm$ and $q$  are written in Duhamel  form as 
\begin{subequations}
\label{integral-system}
\begin{align}
&r(t) = V(t) r_0 + \int\limits_0^t V(t-s) \Big( \mathcal{N}(r, w_\pm) + \f \partial_x \left( |q|^2 \right) \Big) ds \\
& w_\pm (t) = W_\pm (t) w_{\pm, 0} + \int\limits_0^t W_\pm(t-s) \Big( \mathcal{M}_{\pm} \pm 2id\Psi_\pm [P_\pm, r] \partial_x^2 (\Psi_\mp w_\pm) + \f \Psi_\pm \partial_x^2 P_\pm \left( |q|^2 \right) \Big) ds \\
&q(t) = U(t) q_0 + i \beta \int\limits_0^t U(t-s) (qr) ds,
\end{align}
\end{subequations}
where $w_{\pm, 0} := \Psi_\pm \partial_x P_\pm r_0$, 
while the integral equations for the weighted functions are 
\begin{subequations}
\label{integral-system-weighted}
\begin{align}
&x r(t) = V(t) (x r_0) + \int\limits_0^t V(t-s) \Big( (3a\partial_x^2 - 2b \H \partial_x) r+ x \mathcal{N}(r, w_\pm) + \f x \partial_x \left( |q|^2 \right) \Big) ds \\
&x q(t) = U(t) (xq_0) + i \int\limits_0^t U(t-s) (2 \alpha \partial_x q +\beta x qr) ds.
\end{align}
\end{subequations}


\subsection{Estimates on semi-norms}
We introduce the following semi-norms: for any $T>0$ let
\begin{equation}
\label{lambda-seminorms}
\begin{aligned}
&\lambda_1^T (h) = \sup\limits_{0\leq t\leq T} \left\| h(x,t) \right\|_{H_x^1}\, ; \\
&\lambda_2^T (h) = \left \| h(x,t) \right \|_{L_t^6(0, T) L_x^\infty} + \left \| \partial_{x}h(x,t) \right \|_{L_t^6(0, T) L_x^\infty} \, ; \\
&\lambda_3^T (h) = \langle T \rangle^{-1/2}
\left( \sup\limits_{j \in \mathbb{Z}} \int\limits_0^T \int\limits_{\mathbb{R}} \left|\chi_{j/N}(x) \partial_x^2 h(x,t) \right|^2 dx~dt \right)^{1/2}\, ; \\
&\lambda_4^T (h) = \langle T \rangle^{-2}
\left( \sum\limits_{j \in \mathbb{Z}} \sup\limits_{0\leq t \leq T} \sup\limits_{ x \in \mathbb{R}} \left|\chi_{j/N}(x) h(x,t) \right|^2  \right)^{1/2}\, ; \\
& \lambda_5^T (h) = \sup\limits_{0\leq t\leq T} \left\| x h(t) \right\|_{L_x^2}\, ; \\
&\lambda_6^T (h) = \langle T \rangle^{-1/2} \left( \sup\limits_{j \in \mathbb{Z}} \int\limits_0^T \int\limits_{\mathbb{R}} \left|\chi_{j/N} (x) \partial_x \left( x h(x,t) \right) \right|^2 dx~dt \right)^{1/2}\, ;
\end{aligned}
\end{equation}
and
\begin{equation}
\label{mu-seminorms}
\begin{aligned}
&\mu_1^T (h) = \sup\limits_{0\leq t\leq T} \left\| h(t) \right\|_{H_x^3}.\\
\end{aligned}
\end{equation}

Let $Y$ and $Y_T$ be spaces defined as 
\begin{equation*}
Y := \left( H^1(\mathbb{R}) \cap L^2(\mathbb{R}; x^2 dx) \right) \times H^1(\mathbb{R})^2 \times 
\left( H^3(\mathbb{R}) \cap L^2(\mathbb{R}; x^2 dx) \right),
\end{equation*}
and
\begin{equation*}
Y_T := \left\{ (r, w_\pm, q) \in C\big( [0, T]; Y \big) :~ \|(r, w_\pm, q)\|_{Y_T} < \infty \right\},
\end{equation*}
where 
\begin{equation*}
\|(r, w_\pm,q )\|_{Y} := \|r\|_{H_x^1} + \|x r\|_{L_x^2} +
\|w_+\|_{H_x^1} + \|w_-\|_{H_x^1} + \|q\|_{H_x^3} + \|x q\|_{L_x^2}.
\end{equation*}
and
\begin{equation}
\label{seminorms-in-Y_T}
\|(r, w_\pm, q)\|_{Y_T} := \sum_{l=1}^6 \lambda_l^T(r) + \sum_{l=1}^4 \lambda_l^T(w_\pm) + \mu_1^T(q) + \lambda_4^T(q) + \lambda_5^T(q).
\end{equation}

\begin{lemma}
\label{seminorm-bound}
The following estimate holds for the semi-norms in \eqref{lambda-seminorms}--\eqref{mu-seminorms}:
{\small{
\begin{equation}
\label{main-bound}
\begin{aligned}
\|(r, w_\pm, q)\|_{Y_T} &\lesssim  ~  \langle T \rangle \|(r_0, w_{\pm,0}, q_0)\|_{Y} +  \int_0^T \Big( \langle T \rangle \|\mathcal{N}(s)\|_{H_x^1} + \langle T \rangle \|\mathcal{M}_\pm(s)\|_{H_x^1}+ \|x \mathcal{N}(s)\|_{L_x^2}
\Big) ds \qquad\\
& + \int_0^T \Big( \langle T \rangle \|qr(s)\|_{H_x^3}
+ \langle T \rangle \|q^2(s)\|_{H_x^2} + \|x \partial_x (|q|^2) \|_{L_x^2} + \|x qr \|_{L_x^2}
\Big) ds\\
& + \int_0^T \langle T \rangle \Big( \|\Psi_\pm [P_\pm, r] \partial_x^2 (\Psi_\mp w_\pm)\|_{H_x^1} + \|\Psi_\pm \partial_x^2 P_\pm (|q|^2) (s) \|_{H_x^1} \Big) ds \\
& + T \|(r, w_\pm, q)\|_{Y_T}^2 
\end{aligned}
\end{equation}
}}
\end{lemma}

\begin{proof}
We show the steps to derive the bounds for the semi-norms $\lambda_1^T(r), \dots, \lambda_5^T(r)$. The bounds for the rest of the semi-norms in \eqref{seminorms-in-Y_T} are obtained in a similar fashion. 

Using the integral equation for $v$ in \eqref{integral-system} and Minkowski inequality we get
\begin{equation}
\label{lambda-1-bound}
\begin{aligned}
\lambda_1^T(r) &\lesssim \left\|V(t) r_0\right\|_{L_t^\infty (0, T) H_x^1} + \left\|\int_0^t V(t-s) \left(\mathcal{N} (s) + \partial_x(|q^2|) \right) ds \right\|_{L_t^\infty (0, T) H_x^1} \\
& \lesssim \|t_0\|_{H_x^1} + \int_0^T \|\mathcal{N}(s)\|_{H_x^1} ds + \int_0^T \|q^2(s)\|_{H_x^2} ds.
\end{aligned}
\end{equation}
Next, for $l =0, 1$ using Minkowski inequality and Lemma \ref{lemma-semigroup-estimates}, we have
\begin{equation*}
\begin{aligned}
\|\partial_x^l r\|_{L_t^6(0,T) L_x^\infty} & \lesssim \|V(t) \partial_x^l r_0\|_{L_t^6(0,T) L_x^\infty} + \left\|\int_0^t V(t-s) \left(\partial_x^l \mathcal{N} (s) + \partial_x^{l+1}(|q|^2) \right) ds \right\|_{L_t^6 (0, T) L_x^\infty}\\
& \lesssim \| \partial_x^l r_0\|_{L_x^2} + \left\|\int_0^T \|V(t-s) \left(\partial_x^l \mathcal{N} (s) + \partial_x^{l+1}(|q|^2) \right) \|_{L_x^\infty} ds \right\|_{L_t^6 (0, T)}\\
& \lesssim \| r_0\|_{H_x^2} + \int_0^T \|V(t-s) \left(\partial_x^l \mathcal{N} (s) + \partial_x^{l+1}(|q|^2) \right) \|_{L_t^6(0, T) L_x^\infty} ds \\
& \lesssim \| r_0\|_{H_x^2} + \int_0^T \Big( \|\mathcal{N} (s)\|_{H_x^1} + \| q^2 \|_{H_x^2} \Big) ds,
\end{aligned}
\end{equation*}
and this gives the bound for $\lambda_2^T(r)$.

Similarly, for $\lambda_3^T(r)$ we get \begin{equation*}
\begin{aligned}
\langle T \rangle^{1/2} \lambda_3^T(r) &\lesssim  \left( \sup\limits_{j \in \mathbb{Z}} \int\limits_0^T \int\limits_{\mathbb{R}} \left|\chi_j(x) \partial_x^2 V(t) r_0(x,t) \right|^2 dx~dt \right)^{1/2} \\
& \qquad + \left( \sup\limits_{j \in \mathbb{Z}} \int\limits_0^T \int\limits_{\mathbb{R}} \left| \int\limits_0^t  \chi_j(x) \partial_x^2 V(t-s) \Big( \mathcal{N}(r, w_\pm) + \f \partial_x \left( |q|^2 \right) \Big) ds \right|^2 dx~dt \right)^{1/2}\\[3pt]
&:= I_1 + I_2.
\end{aligned}
\end{equation*}
The first term $I_1$ can be easily bounded using the corresponding linear estimate in \eqref{estimates-BO-operators} as 
\begin{equation*}
\begin{aligned}
I_1 \lesssim \langle T \rangle^{1/2} \|r_0\|_{H_x^1}.
\end{aligned}
\end{equation*}
For the second term we use Minkowski inequality for $L_x^2 L_s^1$ to write
\begin{equation*}
\begin{aligned}
&\left(  \int\limits_{\mathbb{R}} \left| \int\limits_0^t  \Big| \chi_j(x) \partial_x^2 V(t-s) \Big( \mathcal{N}(r, w_\pm) + \f \partial_x \left( |q|^2 \right) \Big) \Big| ds \right|^2 dx \right)^{1/2}\\
& \qquad \lesssim \int\limits_0^t  \left( \int\limits_\mathbb{R} \Big|  \chi_j(x) \partial_x^2 V(t-s) \Big( \mathcal{N}(r, w_\pm) + \f \partial_x \left( |q|^2 \right) \Big) \Big|^2 dx \right)^{1/2} ds,
\end{aligned}
\end{equation*}
and then we get
\begin{equation*}
\begin{aligned}
I_2 & \lesssim \left( \sup\limits_{j \in \mathbb{Z}} \int\limits_0^T \left( \int\limits_0^T  \left( \int\limits_\mathbb{R} \Big|  \chi_j(x) \partial_x^2 V(t-s) \Big( \mathcal{N}(r, w_\pm) + \f \partial_x \left( |q|^2 \right) \Big) \Big|^2 dx \right)^{1/2} ds \right)^2 dt \right)^{1/2} \\
& \lesssim \left( \sup\limits_{j \in \mathbb{Z}} \left( \int\limits_0^T \left( \int\limits_0^T   \int\limits_\mathbb{R} \Big|  \chi_j(x) \partial_x^2 V(t-s) \Big( \mathcal{N}(r, w_\pm) + \f \partial_x \left( |q|^2 \right) \Big) \Big|^2 dx~ dt \right)^{1/2} ds \right)^2 \right)^{1/2}\\
& \lesssim \int\limits_0^T \left( \sup\limits_{j \in \mathbb{Z}}   \int\limits_0^T   \int\limits_\mathbb{R} \Big|  \chi_j(x) \partial_x^2 V(t-s) \Big( \mathcal{N}(r, w_\pm) + \f \partial_x \left( |q|^2 \right) \Big) \Big|^2 dx~ dt  \right)^{1/2} ds \\
& \lesssim \langle T \rangle^{1/2} \int\limits_0^T \left( \|\mathcal{N}\|_{H_x^1} + \|q^2\|_{H_x^2} \right) ds,
\end{aligned}
\end{equation*}
where we used Minkowski inequality again for $L_t^2 L_s^1$ and the linear estimate \eqref{estimates-BO-operators} for the last inequality. The above bounds for $I_1$ and $I_2$ lead to the desired bound for $\lambda_3^T(r)$. The bound for $\lambda_4^T(r)$ is derived similarly to the above estimates for $\lambda_3^T(r)$.

For $\lambda_5^T(r)$, we use the integral equation in \eqref{integral-system-weighted} and the Minkowski inequality to write
\begin{equation}
\label{lambda-5-bound}
\begin{aligned}
\lambda_5^T(r) &\lesssim \left\|x r_0\right\|_{L_x^2} + \int_0^T \left( \|r\|_{H_x^2} + \|x \mathcal{N}(s)\|_{L_x^2} +\|x \partial_x\left( |q|^2 \right)\|_{L_x^2} \right) ds.
\end{aligned}
\end{equation}
To get the appropriate bound for the $\sup$ term, we recall \eqref{gauge-function} and \eqref{der-order-reduction-dx2v}, so that 
$$
\begin{aligned}
\|r\|_{H_x^2} \lesssim \|r\|_{H_x^1} + \|w_+\|_{H_x^1} + \|w_-\|_{H_x^1} + \|r\|_{L_x^\infty} \left( \|w_+\|_{L_x^2}+ \|w_-\|_{L_x^2} \right),
\end{aligned}
$$
which implies 
$$
\sup\limits_{0 \leq s \leq T} \|r\|_{H_x^2} \lesssim
\lambda_1^T(r) + \lambda_1^T(w_\pm) + \|(r, w_\pm, q)\|_{Y_T^2}^2. 
$$
Then, using the bounds for $\lambda_1^T$ terms as in \eqref{lambda-1-bound}, the above estimates give
\begin{equation*}
\begin{aligned}
\lambda_5^T(r) &\lesssim \langle T \rangle \left\| (r_0, w_{\pm, 0}, q_0) \right\|_{Y} + \int_0^T T  \left( \|\mathcal{N}(s)\|_{H_x^1} + \|q^2(s)\|_{H_x^2} + \|\mathcal{M}_\pm\|_{H_x^1} \right) ds \\
& \qquad + \int_0^T T \Big( \|\Psi_\pm [P_\pm, r] \partial_x^2 (\Psi_\mp w_\pm)\|_{H_x^1} + \|\Psi_\pm \partial_x^2 P_\pm (|q|^2) (s) \|_{H_x^1} \Big) ds \\ 
& \qquad +\int_0^T \left( \|x \mathcal{N}(s)\|_{H_x^2} +\|x \partial_x \left( |q|^2 \right) \|_{H_x^3} \right) ds + T \| (r, w_\pm, q) \|_{Y_T^2}^2.
\end{aligned}
\end{equation*}

\end{proof}



\subsection{Estimates on time integrals of the nonlinear terms}
Here we provide the bounds on the time intervals appearing on the right-hand side of \eqref{main-bound}.

\begin{lemma}
\label{lemma-estimate-N-M}
There exist positive $T$-independent constants $\theta_1, \theta_2>0$ and an integer $N\geq 2$ such that, for a polynomial $p(y) := y^2 + y^3 + \dots + y^N$, we have 
\begin{equation}
\label{bound-nonlin-polynom}
\int_0^T \Big( \langle T \rangle \|\mathcal{N}(s)\|_{H_x^1} + \langle T \rangle \|\mathcal{M}_\pm(s)\|_{H_x^1}+ \|x \mathcal{N}(s)\|_{L_x^2} 
\Big) ds \lesssim T^{\theta_1} \langle T \rangle^{\theta_2} p \big( \|(r, w_\pm, q)\|_{Y_T} \big).
\end{equation}
\end{lemma}

\begin{proof}
Recall from the construction that all the terms in the polynomials $\mathcal{N}$ and $\mathcal{M}_\pm$ are written in $v$ and $w_\pm$ up to their first derivatives at most. As a result, without loss of generality, we can consider $(\partial_x r) \Psi_- w_+$ and $(\partial_x \Psi_+) (\partial_x^2 P_+ r)$ as generic terms of $\mathcal{N}$ and $\mathcal{M}_\pm$, respectively. Note that $\partial_x^2 P_+ r = \partial_x (\Psi_- w_+)$, so the generic term of $\mathcal{M}_\pm$ can be written as $(\partial_x \Psi_+) \partial_x (\Psi_- w_+)$.

Then, for the $H_x^1$ norm of the generic term of $\mathcal{N}$ integrated over time, we have
\begin{equation*}
\begin{aligned} 
\|(\partial_x r ) \Psi_- w_+\|_{L_s^1 (0, T) H_x^1} \lesssim \|w_+ \Psi_- \partial_x r\|_{L_s^1 (0, T) L_x^2} + \|\partial_x (w_+ \Psi_- \partial_x r)\|_{L_s^1 (0, T) L_x^2}.
\end{aligned}
\end{equation*}
The first term is bounded using the Sobolev embedding as follows:
\begin{equation}
\label{first-term-lemma-nonl}
\begin{aligned}
\|w_+ \Psi_- \partial_x r\|_{L_s^1 (0, T) L_x^2} & \lesssim \|w_+\|_{L_s^\infty (0, T) L_x^\infty} \|\partial_x r\|_{L_s^\infty (0, T) L_x^2} \|\Psi_-\|_{L_s^1 (0, T) L_x^\infty} \\
& \lesssim T \lambda_1^T(w_+) \lambda_1^T(r) \lesssim T \|(r, w_\pm, q)\|_{Y_T}^2
\end{aligned}
\end{equation}
The estimates for the second term are more involved, and we have 
\begin{equation*}
\begin{aligned}
\|\partial_x (w_+ \Psi_- \partial_x r)\|_{L_s^1 (0, T) L_x^2} \lesssim &~ 
\|(\partial_x w_+) (\Psi_- \partial_x r)\|_{L_s^1 (0, T) L_x^2} + 
\| w_+ (\partial_x \Psi_-)  (\partial_x r)\|_{L_s^1 (0, T) L_x^2} \\
& 
+ \| w_+ \Psi_-  (\partial_x^2 r)\|_{L_s^1 (0, T) L_x^2}.
\end{aligned}
\end{equation*}
We only show how to deal with the last term involving the third derivative, and the rest of the terms are bounded in a similar or easier way, using \eqref{gauge-function} when needed. For the last term, using the Cauchy-Schwartz inequality in $s$, we get
\begin{equation}
\label{second-term-lemma-nonl}
\begin{aligned}
\| w_+ \Psi_-  (\partial_x^2 r)\|_{L_s^1 (0, T) L_x^2} & \lesssim T^{1/2} \left( \int_0^T \int_{\mathbb{R}} |w_+~ \partial_x^2 r|^2 dx~ds \right)^{1/2} \\
& \lesssim T^{1/2} \langle T \rangle^{-1} \left( \sum_{j\in \mathbb{Z}} \int_0^T \int_{\mathbb{R}} |\chi_{j/N}(x) w_+~ \partial_x^2 r|^2 dx~ds \right)^{1/2} \\
& \lesssim T^{1/2} \langle T \rangle^{3/2} \lambda_3^T(r) \lambda_4^T (w_+) \lesssim T^{1/2} \langle T \rangle^{3/2} \|(r, w_\pm, q)\|_{Y_T}^2,
\end{aligned}
\end{equation}
where we used \eqref{sum-of-cut-offs}. 


Next, we prove the bound \eqref{bound-nonlin-polynom} for the $\|\mathcal{M}_\pm\|_{H_x^1}$-related terms. 
Dealing with the generic term of $\mathcal{M}_\pm$, we have 
{\small{
\begin{equation*}
\begin{aligned}
&\|(\partial_x \Psi_+) \partial_x (\Psi_- w_+)\|_{L_s^1 (0, T) H_x^1}  \lesssim \|(\partial_x \Psi_+) \partial_x (\Psi_- w_+)\|_{L_s^1 (0, T) L_x^2} + \|\partial_x ((\partial_x \Psi_+) \partial_x (\Psi_- w_+))\|_{L_s^1 (0, T) L_x^2}\\
& \qquad \lesssim \| v \Psi_+  \partial_x (\Psi_- w_+)\|_{L_s^1 (0, T) L_x^2} + \|\partial_x (v \Psi_+ \partial_x (\Psi_- w_+))\|_{L_s^1 (0, T) L_x^2},
\end{aligned}
\end{equation*}
}}
where we used the definition of $\Psi_+$ in \eqref{gauge-function}. 
Here, the first term can be treated similarly to \eqref{first-term-lemma-nonl}, while the estimates for the second term are similar to \eqref{first-term-lemma-nonl} and \eqref{second-term-lemma-nonl}.

Finally, we show the estimates for the generic term of $\|x \mathcal{N}\|_{L_x^2}$. Using the relation $\partial_x(xr) = r + x \partial_x v$, we write 
\begin{equation*}
\begin{aligned}
\|x (\partial_x r) \Psi_- w_+\|_{L_s^1(0, T) L_x^2} \lesssim \|\partial_x (x r) \Psi_- 
w_+\|_{L_s^1(0, T) L_x^2} + \|r \Psi_- w_+\|_{L_s^1(0, T) L_x^2}.
\end{aligned}
\end{equation*}
The last term is treated similarly to \eqref{first-term-lemma-nonl}, while the first term is treated like \eqref{second-term-lemma-nonl} as follows:
\begin{equation*}
\begin{aligned}
\| \partial_x (xr) \Psi_- w_+ \|_{L_s^1 (0, T) L_x^2} & \lesssim T^{1/2} \left( \int_0^T \int_{\mathbb{R}} |w_+~ \partial_x (xr)|^2 dx~ds \right)^{1/2} \\
& \lesssim T^{1/2} \langle T \rangle^{-1} \left( \sum_{j\in \mathbb{Z}} \int_0^T \int_{\mathbb{R}} |\chi_{j/N}(x) w_+~ \partial_x (xr)|^2 dx~ds \right)^{1/2} \\
& \lesssim T^{1/2} \langle T \rangle^{3/2} \lambda_6^T(r) \lambda_4^T (w_+) \lesssim T^{1/2} \langle T \rangle^{3/2} \|(r, w_\pm, q)\|_{Y_T}^2. 
\end{aligned}
\end{equation*}

\end{proof}

\begin{lemma}
\label{lemma-nonl-u}
There exist positive $T$-independent constants $\theta_1, \theta_2>0$ and an integer $N\geq 2$ such that, for a polynomial $p(y) := y^2 + y^3 + \dots + y^N$, we have  
\begin{equation}
\label{bound-nonlin-u}
\begin{aligned}
\int_0^T \Big( \langle T \rangle \|qr(s)\|_{H_x^3}
+ \langle T \rangle \|q^2(s)\|_{H_x^2} + \|x \partial_x (|q|^2) \|_{L_x^2} &+ \|x qr \|_{L_x^2}
\Big) ds \\
&\lesssim T^{\theta_1} \langle T \rangle^{\theta_2} p \big( \|(r, w_\pm, q)\|_{Y_T} \big).
\end{aligned}
\end{equation}
\end{lemma}

\begin{proof}
The last three terms on the left-hand side of \eqref{bound-nonlin-u} are easy to deal with. Indeed, we have
\begin{equation*}
\begin{aligned}
\|q^2\|_{L_s^1(0, T) H_x^2} \lesssim T \|q\|_{L_s^\infty(0, T) H_x^2}^2 \lesssim T \mu_1^T (q)^2 \lesssim T  \|(r, w_\pm, q)\|_{Y_T}^2,
\end{aligned}
\end{equation*}
where we used the Banach algebra property of $H_x^2$ space,
\begin{equation*}
\begin{aligned}
\|x \partial_x (|q|^2)\|_{L_s^1(0, T) L_x^2} \lesssim
T \|x q\|_{L_s^\infty(0, T) L_x^2} \|q\|_{L_s^\infty(0, T) H_x^2} \lesssim 
T \lambda_5^T(q) \mu_1^T(q) \lesssim
T  \|(r, w_\pm, q)\|_{Y_T}^2,
\end{aligned}
\end{equation*}
and similar steps for $\|xqr \|_{L_x^2}$.

When dealing with $\|qr \|_{H_x^3}$, the most problematic term is $q \partial_x^3 r$ which can be treated similarly to \eqref{second-term-lemma-nonl}. That is 
\begin{equation*}
\begin{aligned}
\|q\partial_x^3 r\|_{L_s^1(0, T) L_x^2} \lesssim \|q\partial_x^2(\Psi_- w_+)\|_{L_s^1(0, T) L_x^2} + \|q\partial_x^2(\Psi_+ w_-)\|_{L_s^1(0, T) L_x^2},
\end{aligned}
\end{equation*}
where we decomposed $\partial_x^3 r$ similar to \eqref{der-order-reduction-dx2v}. The bounds for these terms can be derived using the steps as in \eqref{second-term-lemma-nonl}.

\end{proof}

\begin{lemma}
\label{lemma-nonl-commutator}
There exist positive $T$-independent constant $\theta_1. \theta_2>0$ and an integer $N\geq 2$ such that, for a polynomial $p(y) := y^2 + y^3 + \dots + y^N$, we have 
\begin{equation}
\label{bound-nonlin-commut}
\int_0^T \langle T \rangle \Big( \|\Psi_\pm [P_\pm, r] \partial_x^2 (\Psi_\mp w_\pm)\|_{H_x^1} + \|\Psi_\pm \partial_x^2 P_\pm (|q|^2) \|_{H_x^1} \Big) ds \lesssim  T^{\theta_1} \langle T \rangle^{\theta_2} p \big( \|(r, w_\pm, q)\|_{Y_T} \big).
\end{equation}
\end{lemma}

\begin{proof}
The bound for the second term in \eqref{bound-nonlin-commut} can be obtained using the Banach algebra property of the $H_x^3$ space,
\begin{equation*}
\begin{aligned}
&\|\Psi_\pm \partial_x^2 P_\pm (|q|^2) \|_{L_s^1(0, T) H_x^1}  \lesssim \|\Psi_\pm \partial_x^2 P_\pm (|q|^2) \|_{L_s^1(0, T) L_x^2} + \|\partial_x \left( \Psi_\pm \partial_x^2 P_\pm (|q|^2) \right) \|_{L_s^1(0, T) L_x^2}\\
&\qquad\qquad \lesssim \| \partial_x^2 (|q|^2) \|_{L_s^1(0, T) L_x^2} + \| r \partial_x^2 (|q|^2) \|_{L_s^1(0, T) L_x^2} + \| \partial_x^3 (|q|^2) \|_{L_s^1(0, T) L_x^2}\\
& \qquad\qquad\lesssim \| |q|^2 \|_{L_s^1(0, T) H_x^3} \left(1 + \| r\|_{L_s^\infty(0, T) L_x^\infty} \right)\\
& \qquad\qquad\lesssim T \left( \|(r, w_\pm, q)\|_{Y_T}^2 + \|(r, w_\pm, q)\|_{Y_T}^3 \right).
\end{aligned}
\end{equation*}
To treat the first term in \eqref{bound-nonlin-commut} we use the commutator estimates \eqref{commutator-estimate}. 
First, we write
\begin{equation}
\label{commut-term-bounds}
\begin{aligned}
\|\Psi_\pm [P_\pm, r] \partial_x^2 (\Psi_\mp w_\pm)\|_{L_s^1(0, T) H_x^1} \lesssim & ~ \|\Psi_\pm [P_\pm, r] \partial_x^2 (\Psi_\mp w_\pm)\|_{L_s^1(0, T) L_x^2} \\
& + \|\partial_x \big( \Psi_\pm [P_\pm, r] \partial_x^2 (\Psi_\mp w_\pm) \big)\|_{L_s^1(0, T) L_x^2}.
\end{aligned}
\end{equation}
Then, using H\"{o}lder inequality in $s$, \eqref{commutator-estimate} and \eqref{gauge-function}, 
\begin{equation*}
\begin{aligned}
&\|\Psi_\pm [P_\pm, r] \partial_x^2 (\Psi_\mp w_\pm)\|_{L_s^1(0, T) L_x^2}  \lesssim \int_0^T \|\partial_x r \|_{L_x^\infty} \| \partial_x \left(\Psi_\mp w_\pm \right) \|_{L_x^2} ds\\
& \qquad \qquad\lesssim 
T^{5/6} \|\partial_x r\|_{L_s^6 (0, T) L_x^\infty} \left( \|r\|_{L_s^\infty (0, T) L_x^\infty}  \|w_\pm\|_{L_s^\infty (0, T) L_x^2} + \|w_\pm\|_{L_s^\infty (0, T) H_x^1} \right)\\
& 
\qquad\qquad\lesssim T^{5/6} \left( \|(r, w_\pm, q)\|_{Y_T}^2 + \|(r, w_\pm, q)\|_{Y_T}^3 \right).
\end{aligned}
\end{equation*}
For the second term in \eqref{commut-term-bounds}, using \eqref{gauge-function}, we have
\begin{equation*}
\begin{aligned}
\|\partial_x \big( \Psi_\pm [P_\pm, r] \partial_x^2 (\Psi_\mp w_\pm) \big)\|_{L_s^1(0, T) L_x^2} \lesssim & ~ \|\Psi_\pm r [P_\pm, r] \partial_x^2 (\Psi_\mp w_\pm)\|_{L_s^1(0, T) L_x^2} \\
& + \|\Psi_\pm \partial_x [P_\pm, r] \partial_x^2 (\Psi_\mp w_\pm) \|_{L_s^1(0, T) L_x^2}.
\end{aligned}
\end{equation*}
The most difficult term here is the last one, when the derivative hits the commutator part. To treat it we again use \eqref{commutator-estimate} as follows
\begin{equation}
\label{commutator-lemma-term-estimate}
\begin{aligned}
\| \Psi_\pm \partial_x [P_\pm, r] \partial_x^2 (\Psi_\mp w_\pm) \|_{L_s^1(0, T) L_x^2}  & \lesssim \int_0^T \|\partial_x^2 r\|_{L_x^\infty} \|\partial_x (\Psi_\mp w_\pm) \|_{L_x^2} ds\\ 
& \lesssim T^{5/6} \left( \lambda_2^T(w_+) + \lambda_2^T(w_-) \right) \lambda_1^T(w_\pm) \big( 1+ \lambda_1^T(r)\big),
\end{aligned}
\end{equation}
where we used H\"{older} inequality in $s$ and the relation \eqref{der-order-reduction-dx2v}.
\end{proof}

\subsection{Fixed-point argument} 
Applying the results of Lemmas \ref{lemma-estimate-N-M}-\ref{lemma-nonl-commutator} to the estimate \eqref{main-bound}, there exist positive constants $C$, $\theta_1$ and $\theta_2$ such that 
\begin{equation}
\label{inequality-integral-system}
\|(r, w_\pm, q)\|_{Y_T} \leq C\langle T \rangle \|(r_0, w_{\pm,0}, q_0)\|_{Y} + CT^{\theta_1} \langle T \rangle^{\theta_2} p\left( \|(r, w_\pm, q)\|_{Y_T} \right).
\end{equation}
Now the proof of Proposition \ref{well-posedness-proposition} follows from a fixed point argument applied to the integral equations in \eqref{integral-system}. 

We provide the proof for $p(y) = y^2$ in \eqref{inequality-integral-system}. The proof for general case $p(y) = y^2 + y^3 + ... y^N$ follows in the same manner. Starting from  the estimate
\begin{equation}
\label{fixed-point-ineq-2}
\|(r, w_\pm, q)\|_{Y_T} \leq C \langle T \rangle \|(r_0, w_{\pm,0}, q_0)\|_{Y} + CT^{\theta_1} \langle T \rangle^{\theta_2}  \|(r, w_\pm, q)\|_{Y_T}^2, 
\end{equation}
we show that, given $\varepsilon:= \|(r_0, w_{\pm,0}, q_0)\|_{Y}$, there exist $\delta>0$ and $T>0$ such that the map defined by the right-hand side of the integral equations \eqref{integral-system} takes a ball of radius $\delta$, in $Y_T$, which we denote by $B_\delta^{Y_T}$, into itself. 
Moreover, repeating the same arguments one can
verify that the map is a contraction, hence, there is a unique fixed point in $B_\delta^{Y_T}$.
The value of $T$ is not known beforehand, so we show the proof for both possibilities: $T \leq 1$ and $T \geq 1$. 

Assume $T \leq 1$, so that $\langle T \rangle \leq 2$. Then, \eqref{fixed-point-ineq-2} becomes 
\begin{equation}
\label{fixed-point-ineq-3}
\|(r, w_\pm, q)\|_{Y_T} \leq C \|(r_0, w_{\pm,0}, q_0)\|_{Y} + CT^{\theta_1}  \|(r, w_\pm, q)\|_{Y_T}^2 
\end{equation}
for another choice of $C$ and the proof is standard. 


If we assume $T \geq 1$, then $\langle T \rangle \leq 2T$. Then, \eqref{fixed-point-ineq-2} becomes 
\begin{equation}
\label{fixed-point-ineq-4}
\|(r, w_\pm, q)\|_{Y_T} \leq C T \|(r_0, w_{\pm,0}, q_0)\|_{Y} + CT^{\theta}  \|(r, w_\pm, q)\|_{Y_T}^2 
\end{equation}
for another choice of $C$ and $\theta = \theta_1+\theta_2$. 
From \eqref{fixed-point-ineq-4}, we want to show the existence of positive $\delta$ and $T$ such that 
\begin{equation*}
\begin{aligned}
CT\varepsilon \leq \frac{1}{2} \delta \quad \text{and} \quad C T^{\theta} \delta^2 \leq \frac{1}{2}\delta. 
\end{aligned}
\end{equation*}
This can be done with careful computations, and we derive
\begin{equation}
\label{T-delta-choice-2}
\begin{aligned}
T := \varepsilon^{-1/(1+\theta)} (2C)^{-2/(1+\theta)} \quad \text{and} \quad \delta := \varepsilon^{\theta/(1+\theta)} (2C)^{(\theta-1)/(1+\theta)}.
\end{aligned}
\end{equation}
Note that for sufficiently small $\varepsilon>0$, $\delta$ in \eqref{T-delta-choice-2} is sufficiently small too and the bound for $T$ is compatible with the corresponding assumption $T \geq 1$. Moreover, this shows that as $\varepsilon \to 0$, we have $T \to \infty$. 
On the other hand, for sufficiently large $\varepsilon>0$, the bound for $T$ is not compatible with $T \geq 1$ assumption, which guarantees that $T \leq 1$ in such case.

\section*{Acknowledgements}

We thank Didier Pilod for helpful discussions and for pointing out to us several references.
Part of this work was written when A.K. was a Postdoctoral Fellow at the University of Toronto; A.K. was partially supported by University of Toronto and McMaster University.
Part of this work is contained in C.K.'s PhD thesis. C.K is
currently a Coleman Postdoctoral Fellow at Queen's University. 
C. S. is partially supported by the Natural Sciences and Engineering Research Council of Canada (NSERC) under grant RGPIN-2024-03886.

\appendix

\section{Proof of Proposition \ref{cubicH-summary}}
\label{appendix-proof}

\begin{proposition} \label{proposition-term-I}
The cubic part of $\mathrm{I}$ is given by
\begin{equation}
\label{I-3-definition-appendix}
\begin{aligned}
\mathrm{I}^{(3)}
= \frac{1}{2}\int_{\mathbb{R}} &  \Big[ -\rho \eta ~(DB_0^{-1} G_{11}^{(0)}\xi)^2
-(\rho-\rho_1)\eta ~(G^{(0)}B_0^{-1} G_{11}^{(0)}\xi)^2  \\
& \quad + \rho_1 \eta  ~(DB_0^{-1} G^{(0)}\xi)^2
 - \rho_1 \eta_1 ~(G^{(0)}B_0^{-1}G_{12}^{(0)}\xi)^2\Big]\, dx.
\end{aligned}
\end{equation}
\end{proposition}
\begin{proof}
Using the expansions \eqref{DNO-expansion} and \eqref{B-expansion} we get
\begin{align*}
&G_{11}B^{-1}G = (G_{11}^{(0)}+G_{11}^{(10)}+G_{11}^{(01)}) \left (B_{0}^{-1} -  B_{0}^{-1} B^{(1)} B_{0}^{-1} \right )(G^{(0)}+G^{(1)}) + h.o.t. \nonumber \\
& \quad= G_{11}^{(0)} B_{0}^{-1} G^{(0)} + {\rho_{1}G^{(0)}}B_{0}^{-1}(G_{11}^{(10)}+G_{11}^{(01)})B_{0}^{-1} {G^{(0)}} + {\rho G_{11}^{(0)}}{B_{0}^{-1}}G^{(1)} {B_{0}^{-1}} G_{11}^{(0)} + h.o.t.
\end{align*}
From the definitions \eqref{DNO-10-definition}--\eqref{DNO-01-definition}, we have
\begin{equation}
\label{DNO-linear-thru-leading}
\begin{cases}
G_{11}^{(10)} = G_{11}^{(0)} \eta G_{11}^{(0)} - D \eta D, \\
G_{11}^{(01)} = -G_{12}^{(0)} \eta_1 G_{12}^{(0)}.
\end{cases}
\end{equation}
We substitute \eqref{DNO-linear-thru-leading} and  $G^{(1)}$ from \eqref{DNO-1-definition} into the expression for $G_{11}B^{-1}G$ to rewrite it in terms of the leading terms of the Dirichlet-Neumann operators. Then, we put the result back into $I$ in \eqref{kinetic-energy-terms} to obtain \eqref{I-3-definition-appendix}.
\end{proof}

\begin{proposition}\label{proposition-term-II}
The cubic part of $\mathrm{II}$ is given by
\begin{equation}
\begin{aligned}
\mathrm{II}^{(3)} = \int_{\mathbb{R}} \Big[ & - \rho\eta (DB_0^{-1} G_{11}^{(0)} \xi) 
(DB_0^{-1} G_{12}^{(0)} \xi_1) - (\rho-\rho_1) \eta 
(G^{(0)}B_0^{-1} G_{11}^{(0)} \xi) 
( G^{(0)}B_0^{-1} G_{12}^{(0)} \xi_1) \\
& -\rho\eta (D B_0^{-1} G^{(0)} \xi) (DB_0^{-1} G_{12}^{(0)} \xi_1)\\
& - \eta_1 (G_{12}^{(0)} B_0^{-1} G^{(0)} \xi)
\left( G^{(0)} B_0^{-1} (\rho_1 G_{11}^{(0)} + \rho G^{(0)}) \xi_1 \right)\Big]\, dx. 
\end{aligned}    
\end{equation}
\end{proposition}
\begin{proof}
We expand $G(\eta)B^{-1}G_{12}$ in the definition $\mathrm{II}$ in \eqref{kinetic-energy-terms} using \eqref{DNO-expansion} and \eqref{B-expansion}, and obtain
\begin{equation}
\label{term-II-prop-inside-term-expansion}
\begin{aligned}
G(\eta)B^{-1}G_{12}
& = G^{(0)}G_{12}^{(0)}B_{0}^{-1} + (D\eta D - G^{(0)}\eta G^{(0)})B_{0}^{-1}G_{12}^{(0)} + G^{(0)}B_{0}^{-1}G_{11}^{(0)}\eta G_{12}^{(0)} \\
&\quad - G^{(0)}B_{0}^{-1}G_{12}^{(0)}\eta_{1}G_{11}^{(0)} -\rho G^{(0)}B_{0}^{-1}(G_{11}^{(0)}\eta G_{11}^{(0)} - D\eta D)B_{0}^{-1} G_{12}^{(0)} \\
&\quad + \rho G^{(0)}B_{0}^{-1}G_{12}^{(0)}\eta_{1}G_{12}^{(0)}B_{0}^{-1}G_{12}^{(0)} - \rho_{1}G^{(0)}B_{0}^{-1} D\eta D B_{0}^{-1} G_{12}^{(0)} \\
&\quad + \rho_{1} G^{(0)}B_{0}^{-1} G^{(0)}\eta G^{(0)}B_{0}^{-1}G_{12}^{(0)} + h.o.t.
\end{aligned}
\end{equation}
Next we group the first part of the fifth term and the third term in \eqref{term-II-prop-inside-term-expansion} as
\begin{align}
G^{(0)}B_{0}^{-1} G_{11}^{(0)}\eta G_{12}^{(0)} - \rho G^{(0)} B_{0}^{-1} G_{11}^{(0)} \eta G_{11}^{(0)}B_{0}^{-1} G_{12}^{(0)} &= \rho_{1}G^{(0)} B_{0}^{-1} G_{11}^{(0)} \eta G_{12}^{(0)} B_{0}^{-1} G^{(0)},
\end{align}
where we used the definition of $B_0$ in \eqref{B-expansion}. In a like manner, we combine the second part of the second term with the eighth (the last) term to get $-\rho G_{11}^{(0)}B_{0}^{-1} G^{(0)} \eta G^{(0)} B_{0}^{-1} G_{12}^{(0)}$.
We also add the first part of the second term to the seventh term to get $\rho D B_{0}^{-1} G_{11}^{(0)} \eta D B_{0}^{-1} G_{12}^{(0)}$.
These computations result in the $\eta$-related terms in \eqref{term-II-prop-inside-term-expansion}. 

The treatment of $\eta_1$-related terms in \eqref{term-II-prop-inside-term-expansion} is similar, and we use the relation 
\begin{equation}
\label{DNO-quadratic-relation}
\left( G_{11}^{(0)} \right)^{2}-\left( G_{12}^{(0)} \right)^{2} = \left( G^{(0)} \right)^{2}.
\end{equation} 
\end{proof}

\begin{proposition}\label{proposition-term-III}
The cubic part of $\mathrm{III}$ is given by
\begin{equation}
\begin{aligned}
\mathrm{III}^{(3)} =  \frac{1}{2}\int_{\mathbb{R}} \Big[ &- (\rho-\rho_1) \eta 
(G^{(0)}B_0^{-1}G_{12}^{(0)} \xi_1)^2
+ \frac{\rho}{\rho_1}(\rho-\rho_1)\eta
(DB_0^{-1}G_{12}^{(0)} \xi_1)^2
\\
&\quad -\frac{1}{\rho_1} \eta_1 \left( G^{(0)} B_0^{-1} (\rho_1 G_{11}^{(0)} + \rho G^{(0)}) \xi_1 \right) 
-\frac{1}{\rho_1}\eta_1 (D\xi_1)^2 \Big]\, dx.
\end{aligned}
\end{equation}
\end{proposition}
\begin{proof}
The proof follows from steps similar to the proofs of Proposition \ref{proposition-term-I}--\ref{proposition-term-II}.
\end{proof}


The expression \eqref{cubic-Hamiltonian-1} now follows from Propositions \ref{proposition-term-I}-\ref{proposition-term-III}. 

\section{Formulas for coefficients $\kappa_j$}
\label{appendix-coefficients}
We provide the explicit expressions for the  $\kappa$ and $\kappa_j$ coefficients, which appear in the Hamiltonian \eqref{H3-in-tilde-variables} and elsewhere in the paper. For the full derivation of the formulas, we refer the reader to Section 6 of the PhD thesis of the second author \cite{CPAK23}.

The final expression for the coefficient $\kappa$ is
\begin{align} \label{kappa}
\kappa &=  \frac{\rho_{1}}{2\sqrt{g(\rho-\rho_{1})}}(b^{+})^{(0)} (\mathcal{A}_{4}^{(0)})^{2} 
+ \frac{1}{2\rho_{1}\sqrt{g\rho_{1}}}(a^{+})^{(0)} (\mathcal{A}_{5}^{(0)})^{2},
\end{align}
where the constants are obtained from the expansions
\begin{equation}
\label{a-plus-expansion}
\begin{aligned}
a^{+}(\varepsilon D_X)
&= \sqrt{1-\frac{\rho_1}{\rho}} - \varepsilon \frac{\rho_1}{\rho} \sqrt{1-\frac{\rho_1}{\rho}} h_{1} |D_{X}|+ \mathcal{O}(\varepsilon^{2}) \\
&=: (a^{+})^{(0)} + \varepsilon (a^{+})^{(1)}|D_{X}| + \mathcal{O}(\varepsilon^{2}),
\end{aligned}
\end{equation}
$\mathcal{A}_4^{(0)}$ is in \eqref{A4-expansion} and $\mathcal{A}_5^{(0)}$ is found from
\begin{equation}
\label{A5-expansion}
\begin{aligned}
&\mathcal{A}_5(\varepsilon D_X) = - \varepsilon \sqrt{\frac{g\rho_1 (\rho-\rho_1)}{\rho}} \left(1-\varepsilon \frac{\rho_1}{\rho} h_1 |D_X|\right) D_X := \varepsilon \left( \mathcal{A}_5^{(0)} + \varepsilon \mathcal{A}_5^{(1)} |D_X| \right) D_X.
\end{aligned}
\end{equation}
The expression for $\kappa_1$ is given by
\begin{equation}
\begin{aligned} 
\label{kappa 1}
\kappa_{1} &= -\frac{1}{2} \sqrt{\frac{\rho-\rho_{1}}{g}} (b^{+})^{(0)} (\mathcal{B}_{1}^{2}\omega_{1}^{-1})(k_{0}) + \frac{1}{2} \sqrt{\frac{\rho_{1}}{g}} (a^{+})^{(0)} (\mathcal{B}_{2}^{2}\omega_{1}^{-1})(k_{0}) \nonumber \\
&\quad + \frac{\rho}{2\sqrt{g(\rho-\rho_{1})}} (b^{+})^{(0)} (\mathcal{B}_{3}^{2}\omega_{1}^{-1})(k_{0}) - \frac{\rho_{1}}{2\sqrt{g(\rho-\rho_{1})}}(b^{+})^{(0)}(\mathcal{B}_{4}^{2}\omega_{1}^{-1})(k_{0}) \nonumber \\
&\quad - \frac{1}{2\rho_{1}\sqrt{g\rho_{1}}}(a^{+})^{(0)}(\mathcal{B}_{5}^{2}\omega_{1}^{-1})(k_{0}),
\end{aligned}
\end{equation}
where $\mathcal{B}_j$ is defined in \eqref{mathcal-A-B-expressions}, $\omega_1$ is in \eqref{dispersion-relations}, $(a^+)^{(0)}$ is in \eqref{a-plus-expansion} and $(b^+)^{(0)}$ is in \eqref{b-plus-expansion}.
For $\kappa_2$ we use \eqref{a-b-pm-definition}, \eqref{mathcal-A-B-expressions}, \eqref{A4-expansion} and \eqref{A5-expansion} to write
\begin{align} 
\label{kappa 2}
\kappa_{2} &= -\frac{\rho_{1}}{\sqrt{g(\rho-\rho_{1})}} A_{4}^{(0)}(b^{-}\mathcal{B}_{4})(k_{0}) - \frac{1}{\rho_{1}\sqrt{g\rho_{1}}} A_{5}^{(0)}(a^{-}\mathcal{B}_{5})(k_{0}).
\end{align}
For $\kappa_3$ we have
\begin{align} 
\label{kappa 3}
\kappa_{3} &= \frac{\rho_{1}}{\sqrt{g(\rho-\rho_{1})}}(b^{+})^{(0)} \mathcal{A}_{4}^{(0)} \mathcal{A}_{4}^{(1)} + \frac{1}{\rho_{1}\sqrt{g\rho_{1}}} (a^{+})^{(0)}\mathcal{A}_{5}^{(0)}\mathcal{A}_{5}^{(1)}.
\end{align}
For $\kappa_4$ we have
\begin{equation}
\label{kappa 4}
\begin{aligned} 
\kappa_{4} &= - \frac{\sqrt{\rho-\rho_{1}}}{4\sqrt{g}}(b^{+})^{(0)}(\mathcal{B}_{1}^{2}\omega_{1}^{-1})'(k_{0}) + \frac{\sqrt{\rho_{1}}}{4\sqrt{g}}(a^{+})^{(0)}(\mathcal{B}_{2}^{2}\omega_{1}^{-1})'(k_{0}) \\
&\quad + \frac{\rho}{\sqrt{4g(\rho-\rho_{1})}} (b^{+})^{(0)}(\mathcal{B}_{3}^{2}\omega_{1}^{-1})'(k_{0}) - \frac{\rho_{1}}{4\sqrt{g(\rho-\rho_{1})}} (b^{+})^{(0)} (\mathcal{B}_{4}^{2}\omega^{-1})'(k_{0}) \\
&\quad  -\frac{1}{4\rho_{1}\sqrt{g\rho_{1}}} (a^{+})^{(0)}(\mathcal{B}_{5}^{2}\omega_{1}^{-1})'(k_{0}),
\end{aligned}
\end{equation}
while 
\begin{align} 
\label{kappa 5}
\kappa_{5} &= -\frac{\rho_{1}}{2\sqrt{g(\rho-\rho_{1})}}\mathcal{A}_{4}^{(0)}(b^{-}\mathcal{B}_{4})'(k_{0}) - \frac{1}{2\rho_{1}\sqrt{g\rho_{1}}} \mathcal{A}_{5}^{(0)}(a^{-}\mathcal{B}_{5})'(k_{0}),
\end{align}
Next, we write $\kappa_6$ as
\begin{equation}
\label{kappa 6}
\begin{aligned} 
\kappa_{6} &= \frac{-\sqrt{\rho-\rho_{1}}}{2\sqrt{g}}(b^{+})^{(1)}(\mathcal{B}_{1}^{2}\omega_{1}^{-1})(k_{0}) + \frac{\sqrt{\rho_{1}}}{2\sqrt{g}}(a^{+})^{(1)} (\mathcal{B}_{2}^{2}\omega_{1}^{-1})(k_{0}) \\
&\quad + \frac{\rho}{2\sqrt{g(\rho-\rho_{1})}} (b^{+})^{(1)}(\mathcal{B}_{3}^{2}\omega_{1}^{-1})(k_{0}) - \frac{\rho_{1}}{2\sqrt{g(\rho-\rho_{1})}} (b^{+})^{(1)}(\mathcal{B}_{4}^{2}\omega_{1}^{-1})(k_{0}) \\
&\quad - \frac{1}{2\rho_{1}\sqrt{g\rho_{1}}} (a^{+})^{(1)}(\mathcal{B}_{5}^{2}\omega_{1}^{-1})(k_{0}), 
\end{aligned}
\end{equation}
For $\kappa_7$ we use the expansion of $\mathcal{A}_3$ in \eqref{mathcal-A-B-expressions} around $0$ given by 
\begin{equation*}
\mathcal{A}_3(\varepsilon D_X) = \varepsilon^2 \frac{h_1}{\rho} \sqrt{\frac{g\rho_1 (\rho-\rho_1)}{\rho}} |D_X| D_X =: \varepsilon^2 \mathcal{A}_3^{(0)} |D_X| D_X
\end{equation*}
and get
\begin{align} \label{kappa 7}
\kappa_{7} &= \frac{\rho}{\sqrt{g(\rho-\rho_{1})}}\mathcal{A}_{3}^{(0)}(b^{-}\mathcal{B}_{3})(k_{0}) - \frac{\rho_{1}}{\sqrt{g(\rho-\rho_{1})}} \mathcal{A}_{4}^{(1)}(b^{-}\mathcal{B}_{4})(k_{0}) - \frac{1}{\rho_{1}\sqrt{g\rho_{1}}} \mathcal{A}_{5}^{(1)} (a^{-}\mathcal{B}_{5})(k_{0}).
\end{align}
Finally, we write 
\begin{align} \label{kappa 8}
\kappa_{8} &= \frac{\rho_{1}}{2\sqrt{g(\rho-\rho_{1})}}(b^{+})^{(1)}(\mathcal{A}_{4}^{(0)})^{2} + \frac{1}{2\rho_{1}\sqrt{g\rho_{1}}}(a^{+})^{(1)} (\mathcal{A}_{5}^{(0)})^{2}=0,
\end{align}
where the last equality is obtained by direct substitution of constants from \eqref{b-plus-expansion}, \eqref{A4-expansion}, \eqref{a-plus-expansion} and \eqref{A5-expansion} into \eqref{kappa 8}.


\begin{lemma}
\label{lemma-kappa-asymptotics}
Let $\gamma : = 1-\frac{\rho_1}{\rho}$ be sufficiently small. Then, the coefficients $\widetilde \kappa$ and $\widetilde \kappa_j$ in \eqref{kappa-tilde-defn} have the following asymptotics:  

\begin{equation}
\label{kappa-tilde-asymptotics}
\begin{aligned}
&\widetilde \kappa = \frac{g^{1/4}\gamma^{1/4}}{4 h_1^{1/4} \sqrt{2\rho_1}}, \\
& \widetilde \kappa_1 = - \frac{g^{1/4}}{4 h_1^{5/4} \gamma^{3/4} \sqrt{2 \rho_1}} + \mathcal{O}(e^{-1/(8\gamma)}),\quad \quad
\widetilde \kappa_{2} = - \frac{g^{1/4} \gamma^{1/4} h_1^{3/4} (1-\gamma)}{2 \sqrt{2 \rho_1}} + \mathcal{O}(e^{-1/(8\gamma)}),\\
& \widetilde \kappa_{3} = - \frac{g^{1/4}\gamma^{1/4}}{2 h_1^{1/4} \sqrt{2\rho_1}} + \mathcal{O}(e^{-1/(8\gamma)}), \quad  \quad 
\widetilde \kappa_{4} = \frac{(1-\gamma) g^{1/4}}{4 h_1^{1/4} \gamma^{3/4} \sqrt{2 \rho_1}} + \mathcal{O}(e^{-1/(8\gamma)}).
\end{aligned}
\end{equation}
\end{lemma}

\begin{proof}
Note that, from \eqref{c0} and \eqref{k0}, under the regime $\gamma \ll 1$, we have
$$
c_0 = \sqrt{g h_1 \gamma} \ll 1 \quad \text{and} \quad k_0 = \frac{1}{4h_1 \gamma} \gg 1. 
$$
The expressions \eqref{kappa-tilde-asymptotics} are derived from the above expressions $\kappa$ and $\kappa_j$'s using the expansions around $k_0$ of $a^\pm$, $b^\pm$ in \eqref{a-b-pm-definition} and $B_j$ in \eqref{mathcal-A-B-expressions}. These expressions are based on $Q_a, Q_b, Q_c$ from \eqref{A-B-C-coeff}, which in turn, depend on $G^{(0)}, G_{11}^{(0)}, G_{12}^{(0)}$. The latter satisfies the following expansions for $k_0 \gg 1$:
\begin{equation*}
\begin{aligned}
&G_{11}^{(0)}(k_0) = k_0 \coth (h_1k_0) 
= k_0 + \mathcal{O} \left(\frac{1}{\gamma} e^{-1/(2\gamma)}\right),\\
& G_{12}^{(0)}(k_0) = - k_0~ {\rm csch} (h_1 k_0) 
= - \mathcal{O} \left(\frac{1}{\gamma} e^{-1/(4\gamma)}\right).
\end{aligned}
\end{equation*}
\end{proof}






\end{document}